\documentclass[reqno]{amsart}

\usepackage[T1]{fontenc}
\usepackage{lmodern}
\usepackage{amsmath,amsthm,amsfonts,amssymb,mathrsfs,mathtools}
\usepackage{microtype}
\usepackage[shortlabels]{enumitem}
\usepackage{hyperref}
\hypersetup{pdftitle={Prime quadratic-twist L-values and spanning sets of cusp forms},
 pdfauthor={Tianyu Ni},
 colorlinks=false,
 linkcolor=blue}
\setlist[enumerate]{leftmargin=*,label=(\roman*)}
\DeclareMathOperator{\Sl}{SL}
\DeclareMathOperator{\Gl}{GL}
\DeclareMathOperator{\Span}{Span}
\DeclareMathOperator{\Sym}{Sym}
\DeclareMathOperator{\Li}{Li}
\DeclareMathOperator{\re}{Re}
\DeclareMathOperator{\Tr}{Tr}

\DeclareMathOperator{\New}{new}
\DeclareMathOperator{\odd}{odd}
\newcommand{\legendre}[2]{\left(\frac{#1}{#2}\right)}

\newcommand{\EE}{\mathbb E}
\newcommand{\cP}{\mathscr P}
\newcommand{\HH}{\mathcal H}
\newcommand{\CC}{\mathbb C}
\newcommand{\ZZ}{\mathbb Z}
\newcommand{\vect}[1]{\boldsymbol{#1}}
\newtheorem{theorem}{Theorem}[section]
\newtheorem{lemma}[theorem]{Lemma}
\newtheorem{proposition}[theorem]{Proposition}
\newtheorem{corollary}[theorem]{Corollary}
\theoremstyle{definition}
\newtheorem{definition}[theorem]{Definition}

\newtheorem{question}{Question}
\theoremstyle{remark}
\newtheorem{remark}[theorem]{Remark}
\numberwithin{equation}{section}
\date{}

\begin{document}

\title[Prime quadratic-twist $L$-values]{Uniform non-vanishing of prime quadratic twists of standard and symmetric square \texorpdfstring{$L$}{L}-functions
}
\author{Tianyu Ni}
\address{School of Mathematical and Statistical Sciences, Clemson University,
Clemson, SC 29634-0975, USA}
\email{tianyuni1994math@gmail.com}
\subjclass[2020]{11F11, 11F66, 11F67}
\keywords{Quadratic twists; Symmetric square $L$-functions; Twisted periods; Random Euler products}

\begin{abstract}
Let $S_k$ be the space of cusp forms of weight $k$ for $\Sl_2(\mathbb Z)$,
with $\dim S_k\geq1$, and let $\chi_p=\legendre{\cdot}{p}$ for odd primes $p$.
For fixed $s$ with $\re(s)>(k+1)/2$, we prove that the proportion of
primes $p\leq X$ satisfying $L(h,\chi_p,s)=0$ tends to zero as
$X\to\infty$, uniformly over nonzero $h\in S_k$.
We also prove an analogous statement for nontrivial linear combinations of
twisted symmetric square $L$-values of the normalized Hecke eigenforms
when $\re(s)>k$. 
The main ingredient of the proof is a non-concentration theorem for random Euler products under a pairwise local separation condition. 
As applications, we obtain spanning and density-one basis results in $S_k$ for the kernels associated with twisted standard and symmetric square $L$-functions, with explicit basis criteria in specified right half-planes when $\dim S_k=2$ or $3$. We show that almost every tuple of $\dim S_k$ primes gives a basis of $S_k^*$ from twisted periods at every noncentral critical index. We also prove a density-one result for rational bases of $S_k$ obtained from traces of Rankin–Cohen brackets of Eisenstein series at prime levels.

\end{abstract}

\maketitle
\section{Introduction}\label{sect:introduction}

For an even integer $k\geq4$, let $S_k$ be the space of cusp forms of weight
$k$ for $\Sl_2(\ZZ)$. Throughout the paper, we assume that $d:=\dim S_k\geq1$,
and write $\HH_k=\{f_1,\ldots,f_d\}$ for its normalized Hecke eigenforms.
For an odd prime $p$, let $\chi_p=\legendre{\cdot}{p}$. 
Let $f(z)=\sum_{n\geq1}a_f(n)q^n$ be a cusp form in $S_k$, not necessarily a Hecke eigenform. Then one can associate  $f$ with a twisted $L$-function given by the Dirichlet series:
\begin{align*}
L(f,\chi_p,s)=\sum_{n=1}^{\infty}a_f(n)\chi_p(n)n^{-s},\qquad\re(s)\gg0,
\end{align*}
which converges absolutely for $\re(s)>(k+1)/2$ by Deligne's bound $a_f(n)\ll_{f,\epsilon}n^{(k-1)/2+\epsilon}$, and has an analytic continuation to the whole complex plane.
Note that for fixed $s$, each prime $p$ gives a linear functional $h\mapsto L(h,\chi_p,s)$ for $h\in S_k$.
It is therefore natural to ask how often $d$ such evaluations determine
an arbitrary form in $S_k$.

There are many works on determining modular forms by twisted $L$-values. For example, Luo and Ramakrishnan \cite{LuoRamakrishnan1997} proved determination results for elliptic modular forms using twisted critical values; Chinta and Diaconu \cite{ChintaDiaconu2005}
proved a corresponding result for cusp forms on
$\mathrm{GL}_3$; see also the extensions in
\cite{KuanLesesvre2022}. These results concern twisted $L$-values at the center of the critical strip. In this paper, we study twisted \(L\)-values in the region of absolute convergence and give estimates on the proportion of primes $p\leq X$ satisfying $L(h,\chi_p,s)=0$ as $X\to\infty$, uniformly over nonzero $h\in S_k$. These estimates yield density-one results for bases of cusp forms. In dimensions two and three, we also give explicit conditions on the primes that produce such bases.

Let us first recall a motivating result below.  For $b\in\{1,3\}$, put
\begin{align*}
 \cP_{4,b}&:=\{p:p\text{ is prime and }p\equiv b\pmod4\},\\
 \cP_{4,b}(X)&:=\{p\leq X:p\text{ is prime and }p\equiv b\pmod4\},\\
 P_b(X)&:=\#\cP_{4,b}(X).
\end{align*}
Note that $P_b(X)\rightarrow\infty$ as $X\rightarrow\infty$ by Dirichlet's theorem on arithmetic progressions.
In previous work \cite[Theorem 3.2]{Ni2026}, it was shown that, for a cusp form $f\in S_k$ and a complex number $s$ with
$\re(s)>(k+1)/2$,
\begin{align}\label{eq:previous}
 L(f,\chi_p,s)=0\quad\text{for all }p\equiv b\pmod4
 \qquad\Longrightarrow\qquad f=0.
\end{align}
Our first main result shows that the primes at which a twisted value vanishes have relative density zero, uniformly over all nonzero cusp forms in $S_k$, thus strengthening \cite[Theorem 3.2]{Ni2026}.   Throughout the paper, for $G(X)>0$, the notation $F(X)=o(G(X))$ means
$F(X)/G(X)\to0$ as $X\to\infty$.

\begin{theorem}\label{thm:standard}
Fix $b\in\{1,3\}$ and $s\in\CC$ such that $\re(s)>(k+1)/2$. Then
\begin{align}\label{eq:standarduniform}
 \sup_{0\neq h\in S_k}
 \#\{p\in\cP_{4,b}(X):L(h,\chi_p,s)=0\}=o(P_b(X)).
\end{align}
Equivalently, for any $\epsilon>0$, there exists some $X_0$ independent of $h$, such that 
$$\#\{p\in\cP_{4,b}(X):L(h,\chi_p,s)=0\}<\epsilon P_b(X)$$
for every $X\geq X_0$ and every $0\neq h\in S_k$.
\end{theorem}

Our second result concerns twisted symmetric square $L$-functions. Let $f(z)=\sum_{n\geq1}a_f(n)q^n$ be a normalized Hecke eigenform in $S_k$. For each prime $r$, we denote by $\alpha_f(r)$ and $\beta_f(r)$ the roots of 
$$X^2-a_f(r)X+r^{k-1}=0.$$
Following \cite{Shimura1975},
we define the twisted symmetric square $L$-function by
\begin{align}\label{eq:symdefinition}
 L(\Sym^2f,\chi_p,s)
 =&\prod_r\prod_{j=0}^2
 (1-\chi_p(r)\alpha_f(r)^{2-j}\beta_f(r)^j r^{-s})^{-1},
\end{align}
which converges absolutely for $\re(s)>k$. For
$h=\sum_{i=1}^dc_if_i\in S_k$, define the linear functional
\begin{align}\label{eq:symfunctional}
 \mathcal L^{(2)}_{p,s}(h)
:=\sum_{i=1}^dc_iL(\Sym^2f_i,\chi_p,s).
\end{align}
Recall that the critical values (right of the central point) for $L(\Sym^2f,\chi_p,s)$ are at the points $s\in\{k,\ldots,2k-2\}$ with $(-1)^s=\chi_p(-1)$; see \cite[Proposition 2.2.1]{LoefflerIwasawa2019}. A natural question is whether, for a fixed prime $p$, these linear functionals for different $s$ at the critical points span the whole space $S_k^{\ast}$.   For example, \cite{NiXueSymmetricPeriods2025,NiXueperiodandsymperiod2026} studied this question for the non-twisted case in level one, and \cite{Ni2026archive} considered the non-twisted case in square-free levels.   While the question remains open (see e.g., \cite[Conjecture 6.1]{NiXueSymmetricPeriods2025}), Corollary \ref{cor:determinant} to the following theorem presents another aspect where $s$  is fixed with $\re(s)>k$ but $p$ varies, in which case almost every ordered $d$-tuple of primes in $\mathscr{P}_{4,b}$ gives a basis $\mathcal L^{(2)}_{p_1,s},\ldots,\mathcal L^{(2)}_{p_d,s}$ of $S_k^{\ast}$. 

\begin{theorem}\label{thm:symmetric}
Fix $b\in\{1,3\}$ and $s\in\CC$ such that $\re(s)>k$. Then
\begin{align}\label{eq:symuniform}
 \sup_{\vect c\in\CC^d\setminus\{0\}}
 \#\left\{p\in\cP_{4,b}(X):
 \sum_{i=1}^dc_iL(\Sym^2f_i,\chi_p,s)=0\right\}=o(P_b(X)).
\end{align}
Equivalently, for any $\epsilon>0$, there exists some $X_0$ independent of $\vect{c}=(c_1,\ldots,c_d)$, such that 
$$\#\left\{p\in\cP_{4,b}(X):
\sum_{i=1}^dc_iL(\Sym^2f_i,\chi_p,s)=0\right\}<\epsilon P_b(X)$$
for every $X\geq X_0$ and every $\vect c\in\CC^d\setminus\{0\}$.
\end{theorem}

For convenience, write
\begin{align*}
 L^{(1)}(f,\chi,s)&=L(f,\chi,s),&
 L^{(2)}(f,\chi,s)&=L(\Sym^2f,\chi,s),\\
 \sigma_1&=\frac{k+1}{2},&\sigma_2&=k.
\end{align*}
The uniformity in the preceding theorems gives the following result.

\begin{corollary}\label{cor:determinant}
Fix $b\in\{1,3\}$, $m\in\{1,2\}$, and $s\in\CC$ with
$\re(s)>\sigma_m$. Then
\begin{align}\label{eq:detcount}
 \#\left\{(p_1,\ldots,p_d)\in\cP_{4,b}(X)^d:
 \det(L^{(m)}(f_i,\chi_{p_j},s))_{1\leq i,j\leq d}=0\right\}
 =o(P_b(X)^d).
\end{align}
Moreover, outside a set of $o(P_b(X)^d)$ tuples, the determinants are simultaneously nonzero for any fixed finite collection of pairs $(m,s)$ with $m\in\{1,2\}$ and $\re(s)>\sigma_m$.
\end{corollary}

\begin{remark}\label{rem:scope}
All the results above also hold with $\cP_{4,b}(X)$ replaced by $\mathcal{P}_{\odd}(X)$ the set of odd primes up to $X$. 
All limits are taken with the weight, normalized Hecke eigenforms, and evaluation points fixed.
\end{remark}

The proof is based on a non-concentration principle for random Euler products.   We first  introduce the related probability space. Let
$\Omega=\{-1,1\}^{\{\text{primes}\}}$ and $\varepsilon_r(\omega):=\omega_r$
 for $\omega=(\omega_r)_r\in\Omega.$
Define the probability space $(\Omega,\mathcal F,\mu)$
as follows. Let
$\mathcal F=\sigma(\varepsilon_r:r\text{ prime})$
be the $\sigma$-algebra generated by the events
$\{\varepsilon_r=1\}$, and let $\mu$ be the probability measure
given by
\begin{align*}
 \mu\left(\left\{\omega\in\Omega:
 \varepsilon_{r_j}(\omega)=\eta_j\text{ for }1\leq j\leq t\right\}\right)
 =2^{-t}
\end{align*}
for every finite collection of distinct primes
$r_1,\ldots,r_t$ and every
$(\eta_1,\ldots,\eta_t)\in\{-1,1\}^t$.
Thus, the variables $\varepsilon_r$ are independent, each taking
the values $1$ and $-1$ with probability $1/2$.

\begin{theorem}[Non-concentration for Euler products]
\label{thm:abstract}
Let $d\geq1$ and let $u_i(r)\in[-2,2]$ for every prime $r$ and
$1\leq i\leq d$. Suppose that, for each $i\neq j$, we have
\begin{align}\#\{r~{\rm prime}:u_i(r)\neq u_j(r)\}=\infty.\label{eq:abstractseparation}\end{align} Fix
$w\in\CC$ with $\re(w)>1$. Define
\begin{align*}
 Z_i&=\prod_r(1-u_i(r)\varepsilon_r r^{-w}+r^{-2w})^{-1},\\
 Y_i(p)&=\prod_r(1-u_i(r)\chi_p(r)r^{-w}
+\chi_p(r)^2r^{-2w})^{-1},
\end{align*}
where the products are over
all primes, and $r^{-w}:=\exp(-w\log r)$. Then these products converge absolutely and are nonzero. For every
$\vect c=(c_1,\ldots,c_d)\in\CC^d\setminus\{0\}$,
\begin{align}
 \mu\left(\sum_{i=1}^d c_i Z_i=0\right)=0. \label{eq:randomsignP=0}
\end{align}
Moreover, for each $b\in\{1,3\}$,
\begin{align}
 \sup_{\vect c\neq0}
 \#\left\{p\in\cP_{4,b}(X):\sum_{i=1}^d c_iY_i(p)=0\right\}
 =o(P_b(X)). \label{eq:twistP=0}
\end{align}
\end{theorem}

We now sketch the idea of the proof of Theorem \ref{thm:abstract}.
The key observation underlying this probabilistic theorem is, in fact, from the theory of modular forms. Ramakrishnan's result \cite{DukeKowalski2000} recalled in Proposition~\ref{prop:separationthm} implies that if $f,g\in\HH_k$ are distinct, then $a_f(r)^2\neq a_g(r)^2$ for infinitely many
primes $r$; see Lemma \ref{lem:separation}.
Writing $\lambda_f(r)=a_f(r)r^{-(k-1)/2}$, Lemma \ref{lem:separation} shows that both choices
$u_i(r)=\lambda_{f_i}(r)$ and
$u_i(r)=\lambda_{f_i}(r)^2-2$ satisfy \eqref{eq:abstractseparation}.  These choices correspond to the standard and symmetric square families, respectively. Deligne's bound \eqref{eq:deligne} ensures that both choices of
$u_i(r)$ lie in $[-2,2]$, as required in Theorem \ref{thm:abstract}.

The assumption \eqref{eq:abstractseparation} allows us to prove \eqref{eq:randomsignP=0} as follows.  Suppose $d\geq2$ and fix $R>0$.
We can choose pairwise distinct primes $r_{ij}>R$ such that
$u_i(r_{ij})\neq u_j(r_{ij})$ for every $i<j$.
Set $B=\{r_{ij}:i<j\}$, so $|B|=\binom d2$, and put
\[
F_i=\prod_{r\in B}
(1-u_i(r)\varepsilon_r r^{-w}+r^{-2w})^{-1}.
\]
Lemma~\ref{lem:tensors} shows that the $F_i$ are linearly independent
as functions of $(\varepsilon_r)_{r\in B}$.
After fixing all signs outside $B$, we have $Z_i=D_iF_i$,
where each $D_i$ is a nonzero constant by
Lemma~\ref{lem:uniformproduct}.
For fixed coefficients $c_i$ that are not all zero, at least one
of the $2^{|B|}$ equally likely choices of $(\varepsilon_r)_{r\in B}$
therefore makes $\sum_i c_iZ_i$ nonzero.
Its conditional probability of vanishing is thus at most
$1-2^{-\binom d2}$.
Since every prime in $B$ exceeds $R$, the tower property gives
the same bound after conditioning on $(\varepsilon_r)_{r\leq R}$.
As $R\to\infty$, these conditional probabilities converge almost
surely to the indicator of the vanishing event. This indicator
is therefore strictly less than $1$ almost surely, proving
\eqref{eq:randomsignP=0}; see
Proposition~\ref{prop:non-concentration}.

To pass to prime twists, let $\nu$ be the distribution of $Z$ and
$\nu_X$ the distribution of $Y(p)$ when $p$ is chosen uniformly
from $\cP_{4,b}(X)$; see \eqref{eq:distributionofZ} and
\eqref{eq:distributionofY(p)}.
Lemma~\ref{lem:signs} shows that, at any fixed finite set of primes,
the joint distribution of the character values approaches that
of the corresponding independent variables $\varepsilon_r$.
The uniform tail estimate in Lemma~\ref{lem:uniformproduct}
then extends this convergence from the finite Euler products
to the full products, giving $\nu_X\Longrightarrow\nu$.
By \eqref{eq:randomsignP=0}, we have $\nu(H)=0$ for every complex
linear hyperplane $H$ through the origin.
Lemma~\ref{lem:compactness} therefore yields
$\sup_H\nu_X(H)\to0$, proving \eqref{eq:twistP=0}.

The paper is organized as follows. In Section~\ref{sect:preliminaries}, we give the arithmetic preliminaries. Sections~\ref{sect:random} and \ref{sect:proofs} prove Theorem \ref{thm:abstract} and the main density results, Theorems \ref{thm:standard} and \ref{thm:symmetric}. Sections \ref{sect:kernels} and \ref{sect:applications} are dedicated to the applications of Theorems \ref{thm:standard} and \ref{thm:symmetric}. We recall the kernel functions for the twisted $L$-function and the twisted symmetric square $L$-function, thereby giving spanning families of the space of level-one cusp forms $S_k$ that differ from the classical ones. In particular, the traces of the Rankin-Cohen bracket of Eisenstein series at prime levels constructed in \cite{Ni2026} give rational bases  for a density-one set of admissible prime tuples; see Corollary \ref{cor:rationalbases}. In Section \ref{sect:explicitbases}, we make explicit the conditions on the primes that give a basis for the cases $d=2,3$; see Theorem \ref{thm:explicitbases}. Finally, in Section \ref{sect:discussion}, we conclude our paper by proposing several questions for future work.

\section{Preliminaries}\label{sect:preliminaries}

\subsection{Euler products}
Let $f(z)=\sum_{n\geq1}a_f(n)q^n\in\HH_k$. By Deligne's theorem \cite{Deligne1974},
\begin{align}\label{eq:deligne}
 |\alpha_f(r)|=|\beta_f(r)|=r^{(k-1)/2}
\end{align}
for every prime $r$. Note that
$a_f(r)$ is real. Define
\begin{align}\label{eq:normalizedcoeff}
 \lambda_f(r)=\frac{a_f(r)}{r^{(k-1)/2}},
 \qquad b_f(r)=\lambda_f(r)^2-2.
\end{align}
Then $\lambda_f(r),b_f(r)\in[-2,2]$.

For $\re(s)>(k+1)/2$, put $w=s-(k-1)/2$. The Euler product of $L(f,\chi_p,s)$ is
\begin{align}\label{eq:standardproduct}
 L(f,\chi_p,s)
 =\prod_r(1-\lambda_f(r)\chi_p(r)r^{-w}
                 +\chi_p(r)^2r^{-2w})^{-1}.
\end{align}
For $\re(s)>k$, put $w=s-k+1$.  Note that \begin{align}\alpha_f(r)\beta_f(r)=r^{k-1},\quad\alpha_f(r)^2+\beta_f(r)^2=a_f(r)^2-2r^{k-1}.\label{eq:coeffquadratic}\end{align} 
From \eqref{eq:symdefinition} and \eqref{eq:coeffquadratic}, we obtain
\begin{align}\label{eq:reducedsym}
 L(\Sym^2f,\chi_p,s)
 =&
 L(w,\chi_p)\prod_r(1-b_f(r)\chi_p(r)r^{-w}
                 +\chi_p(r)^2r^{-2w})^{-1}.
\end{align}
where $L(s,\chi_p)$ denotes the Dirichlet $L$-function, and $L(w,\chi_p)\neq0$ since $\re(w)>1$.  Note that all local factors in the
above Euler products are $1$ at $r=p$.

\subsection{Separation of Hecke eigenvalues}
We recall the corollary to Theorem~A in Ramakrishnan's appendix to
\cite{DukeKowalski2000}. Write $\mathcal H_k^{\New}(N,\omega)$
for the set of normalized newforms of weight $k$, level $N$,
and character $\omega$.
\begin{proposition}\label{prop:separationthm}
Let $f\in\mathcal H_k^{\New}(N,\omega)$ and
$g\in\mathcal H_k^{\New}(N',\omega)$ have the same weight and
character. Suppose that
\begin{align*}
 a_f(p)^2=a_g(p)^2
\end{align*}
for all primes $p$ outside a set $S$ of Dirichlet density
$\delta(S)<1/18$. Then there exists a quadratic character $\chi$
of conductor dividing $NN'$ such that
\begin{align*}
 a_f(p)=a_g(p)\chi(p)\qquad(p\nmid NN').
\end{align*}
If, in addition, $\omega$ is trivial and $N,N'$ are squarefree,
then $f=g$.
\end{proposition}

We have the following immediate consequence of Ramakrishnan's theorem.  

\begin{lemma}\label{lem:separation}
If $f,g\in\HH_k$ are distinct, then
\begin{align*}
 a_f(r)^2\neq a_g(r)^2
\end{align*}
for infinitely many primes $r$. Consequently, each of the inequalities
$\lambda_f(r)\neq\lambda_g(r)$ and $b_f(r)\neq b_g(r)$ holds for
infinitely many primes.
\end{lemma}
\begin{proof}
Suppose that $a_f(r)^2=a_g(r)^2$ for all but finitely many primes $r$.
Then we must have $f=g$ by Proposition \ref{prop:separationthm}, a contradiction. The second assertion is from \eqref{eq:normalizedcoeff}.
\end{proof}

\subsection{Distribution of quadratic character values}
For positive integers $m$ and $a$ such that $\gcd(m, a)=1$, let $\pi(X;m, a)$ be the counting function
\begin{align*}
    \pi(X;m, a):=\#\{ p\leq X~|~p\equiv a\pmod m\}.
\end{align*}
Recall the Siegel-Walfisz theorem (see e.g., \cite[p.~124]{IwaniecKowalski2004}):
\begin{align}\pi(X;m,a)=\frac{\Li(X)}{\varphi(m)}+o(\Li(X)), \label{eq:SiegelWalfisz}
\end{align}
where $\Li$ denotes the logarithmic integral, $\Li(X)\sim X/\log X$, and $o(\Li(X))/\Li(X)\rightarrow 0$ as $X\rightarrow\infty$. We need the following simple counting lemma. 
\begin{lemma}\label{lem:signs}
Fix $b\in\{1,3\}$, distinct primes $r_1,\ldots,r_t$, and signs
$\eta_1,\ldots,\eta_t\in\{-1,1\}$. Then
\begin{align}\label{eq:signequidistribution}
 \lim_{X\longrightarrow\infty}
 \frac{\#\{p\in\cP_{4,b}(X):\chi_p(r_j)=\eta_j\ (1\leq j\leq t)\}}
      {P_b(X)}=2^{-t}.
\end{align}
\end{lemma}
\begin{proof}
Let $N(X)$ denote the numerator in \eqref{eq:signequidistribution}. We may ignore the finitely many primes $p$ among  $r_1,\dots,r_t$ since they don't affect the limit. 
For an odd $r_j$, quadratic reciprocity gives
\begin{align*}
    \chi_p(r_j)=(-1)^{\frac{r_j-1}{2}\frac{p-1}{2}}\legendre{p}{r_j}.
\end{align*}
Since $p\equiv b\pmod 4$, $(-1)^{\frac{r_j-1}{2}\frac{p-1}{2}}$ is fixed. Thus, to make $\chi_p(r_j)=\eta_j$, $p$ must lie in a residue class modulo $r_j$ such that $\legendre{\cdot}{r_j}$ has the required sign. Note that there are $(r_j-1)/2$ such nonzero residue classes modulo $r_j$.
For $r_j=2$, we have 
\begin{align*}
    \chi_p(2)=(-1)^{\frac{p^2-1}{8}}=\begin{cases}1& p\equiv1,7\pmod 8,\\ -1&p\equiv 3,5\pmod 8.\end{cases}
\end{align*}
Note that $p\equiv b\pmod 4$ gives $p\equiv b\pmod8$ or $p\equiv b+4 \pmod 8$, and each residue class gives different values of $\chi_p(2)$. Thus, to make $\chi_p(2)=\eta_j$, $p$ needs to lie in exactly one residue class modulo $8$.

Now, we write $\delta=1$ if some $r_j=2$ and $\delta=0$ otherwise. Then $t=\delta+\#\{j:r_j\neq2\}$. Let
$M=8\prod_{r_j\neq 2}r_j.$
By the Chinese remainder theorem and the argument above, there are
\begin{align*}
    2^{1-\delta}\prod_{r_j\neq2}\frac{r_j-1}{2}=\frac{\varphi(M)}{2^{t+1}}
\end{align*}
different residue classes modulo $M$ satisfying all conditions. Denote by $\mathcal{C}$ the set of these $\varphi(M)/2^{t+1}$ residue classes modulo $M$. Summing over  $\mathcal{C}$ and applying \eqref{eq:SiegelWalfisz}, we get 
\begin{align*}
N(X)&=\sum_{a\in\mathcal{C}}\pi(X;M,a)=|\mathcal{C}|\left(\frac{\Li(X)}{\varphi(M)}+o(\Li(X))\right)=2^{-(t+1)}\Li(X)+o(\Li(X)),
\end{align*}
which implies that
\begin{align*}
\lim_{X\rightarrow\infty}\frac{N(X)}{P_b(X)}&=\lim_{X\rightarrow\infty}\frac{2^{-(t+1)}\Li(X)+o(\Li(X))}{\frac{\Li(X)}{\varphi(4)}+o(\Li(X))}=2^{-t},
\end{align*}
as desired.
\end{proof}

\section{A non-concentration result for Euler products}\label{sect:random}

\subsection{Uniform convergence of the Euler products}
Let $u_i(r)\in[-2,2]$, where $r$ runs over the primes and
$1\leq i\leq d$, and assume that
\begin{align}
 \#\{r:u_i(r)\neq u_j(r)\}=\infty\qquad(i\neq j).
\end{align}
Fix $w\in\CC$ with $\sigma:=\re(w)>1$. We first prove a convergence
statement that applies to both $Z_i$ and $Y_i(p)$ in Theorem~\ref{thm:abstract}.
For $e\in\{-1,0,1\}$, put
\begin{align}\label{eq:localV}
 V_{i,r}(e)=(1-u_i(r)e r^{-w}+e^2r^{-2w})^{-1}.
\end{align}
For $\vect e=(e_r)_r\in\{-1,0,1\}^{\{\text{primes}\}}$, write
\begin{align}\label{eq:abstractY}
 Y_i(\vect e)=\prod_rV_{i,r}(e_r),\qquad
 Y(\vect e)=(Y_1(\vect e),\ldots,Y_d(\vect e)).
\end{align}
Then taking $e_r=\varepsilon_r$ gives the random
variable $Z_i$ in Theorem~\ref{thm:abstract}, since $\varepsilon_r^2=1$; 
taking $e_r=\chi_p(r)$ gives $Y_i(p)$. Here, we allow
the value $0$ because $\chi_p(p)=0$. Note that $V_{i,r}(0)=1$, so
setting a coordinate equal to zero simply omits the corresponding
Euler factor.

The following lemma gives a tail estimate independent of the choice
of $\vect e$. Its proof uses only $|u_i(r)|\leq2$ and $\sigma>1$.

\begin{lemma}\label{lem:uniformproduct}
The products in \eqref{eq:abstractY} converge absolutely and uniformly
in $\vect e$, and their values are nonzero. The set
\begin{align}
 K=Y\left(\{-1,0,1\}^{\{{\rm primes}\}}\right) \label{eq:compactK}
\end{align}
is a compact subset of $(\CC^\times)^d$. Moreover, for $R\geq2$,
\begin{align}\label{eq:uniformtail}
 \sup_{\vect e}\left|Y_i(\vect e)-\prod_{r\leq R}V_{i,r}(e_r)\right|
 \ll_w\sum_{r>R}r^{-\sigma}
 \qquad(1\leq i\leq d).
\end{align}
\end{lemma}
\begin{proof}
Since $u_i(r)\in[-2,2]$, write $u_i(r)=2\cos\theta_{i,r}$ with
$\theta_{i,r}\in\mathbb R$. Then
\begin{align*}
 1-u_i(r)e r^{-w}+e^2r^{-2w}
 =(1-e\exp(\mathrm{i}\theta_{i,r})r^{-w})
  (1-e\exp(-\mathrm{i}\theta_{i,r})r^{-w}),
\end{align*}
where both factors are nonzero, since $|e r^{-w}|\leq r^{-\sigma}<1$.
The triangle inequality gives
\begin{align*}
 (1+r^{-\sigma})^{-2}\leq |V_{i,r}(e)|
 \leq(1-r^{-\sigma})^{-2}.
\end{align*}
Put $a_{i,r}(e)=V_{i,r}(e)-1$. Then
\begin{align*}
 |a_{i,r}(e)|
 \leq\frac{2r^{-\sigma}+r^{-2\sigma}}{(1-r^{-\sigma})^2}
 \leq C_\sigma r^{-\sigma},
 \qquad
 C_\sigma:=\frac{2+2^{-\sigma}}{(1-2^{-\sigma})^2},
\end{align*}
uniformly in $i$ and $e$.
Since $\sigma>1$, we have
\begin{align*}
 \sum_r r^{-\sigma}\leq\sum_{n=2}^{\infty}n^{-\sigma}<\infty,
 \qquad
 \sum_r|a_{i,r}(e_r)|
 \leq C_\sigma\sum_r r^{-\sigma}=:A_\sigma<\infty.
\end{align*}
Thus, $Y_i(\vect e)=\prod_r(1+a_{i,r}(e_r)) $ converges absolutely, independently of $\vect e$. The non-vanishing of $Y_i(\vect e)$ is from the fact that
\begin{align}\label{eq:productbounds}
|Y_i(\vect e)|&\geq \prod_r(1+r^{-\sigma})^{-2}>0.
\end{align}

We now prove the uniform  estimate \eqref{eq:uniformtail}. Note that for any finite set $B$ of primes, we have
\begin{align*}\left|\prod_{r\in B} V_{i,r}(e_r)\right|&\leq\prod_{r\in B}|1+a_{i,r}(e_r)|\\&\leq\exp\left(\sum_{r\in B}|a_{i,r}(e_r)|\right)\\&\leq \exp(A_{\sigma}).\end{align*}
It follows that for $S>R$, we have
\begin{align*}
 \left|\prod_{r\leq S}V_{i,r}(e_r)
       -\prod_{r\leq R}V_{i,r}(e_r)\right|
 &\leq \exp(A_{\sigma})\sum_{R<r\leq S}|a_{i,r}(e_r)|\\
 &\leq \exp(A_{\sigma})\cdot C_\sigma\sum_{R<r\leq S}r^{-\sigma},
\end{align*}
where the final sum tends to zero
as $R\to\infty$, independently of $S$ and $\vect e$. Hence the
partial products $\prod_{r\leq R}V_{i,r}(e_r)$ form a uniformly Cauchy family. Letting $S\to\infty$
proves uniform convergence and \eqref{eq:uniformtail}.

Finally, give $\{-1,0,1\}^{\{\text{primes}\}}$ the product topology,
with each factor discrete. Then $\{-1,0,1\}^{\{\text{primes}\}}$ is compact by Tychonoff's Theorem. Each finite Euler product $\prod_{r\leq R}V_{i,r}(e_r)$ depends on only
finitely many coordinates and is therefore continuous. Uniform
convergence shows that $Y$ is continuous. Since the domain of $Y$ is compact, the image $K$ is compact. Finally, by \eqref{eq:productbounds},
$K\subseteq(\CC^\times)^d$. This completes the proof.
\end{proof}
\subsection{Random signs and conditional probabilities}
Define
\begin{align*}
 Z_i(\omega)=\prod_r V_{i,r}(\varepsilon_r(\omega)),
 \qquad Z(\omega):=(Z_1(\omega),\ldots,Z_d(\omega)).
\end{align*}
Then $Z$ is a random vector on $(\Omega,\mathcal F,\mu)$.
Its distribution on $\CC^d$ is the probability measure
\begin{align}
 \nu(E):=\mu(\{\omega\in\Omega:Z(\omega)\in E\})
 \qquad(E\subseteq\CC^d\text{ Borel}). \label{eq:distributionofZ}
\end{align}
By Lemma~\ref{lem:uniformproduct}, $Z(\omega)\in K$ for every
$\omega\in\Omega$, so $\nu(K)=1$.

For an event $A\in\mathcal F$ and a sub-$\sigma$-algebra
$\mathcal G\subseteq\mathcal F$, write
$$\mu(A\mid\mathcal G)=\EE(\mathbf1_A\mid\mathcal G),$$
which is the $\mathcal G$-measurable random variable characterized,
up to almost sure equality, by
\begin{align*}
 \int_C\mu(A\mid\mathcal G)\,d\mu=\mu(A\cap C)
 \qquad(C\in\mathcal G).
\end{align*}
Here $\mathbf1_A$ is the indicator of $A$ and $\EE$ denotes expectation.
Recall the tower property: if $\mathcal F'\subseteq\mathcal G$, then
\begin{align}\label{eq:tower}
 \EE(\mu(A\mid\mathcal G)\mid\mathcal F')
 =\mu(A\mid\mathcal F').
\end{align}

\begin{lemma}\label{lem:blockconditioning}
Let $B$ be a finite set of primes and
$\mathcal G_B=\sigma(\varepsilon_r:r\notin B)$ be the $\sigma$-algebra generated by the events $\{\varepsilon_r=1\}$ for $r\notin B$.
For $\omega\in\Omega$ and $\eta\in\{-1,1\}^B$, let
$\omega^{B,\eta}$ be obtained by replacing the coordinates of $\omega$
in $B$ by $\eta$. Then, for every $A\in\mathcal F$,
\begin{align}\label{eq:blockconditioning}
 \mu(A\mid\mathcal G_B)(\omega)
 =2^{-|B|}\sum_{\eta\in\{-1,1\}^B}
   \mathbf1_A(\omega^{B,\eta})
 \qquad\text{almost surely}.
\end{align}
\end{lemma}
\begin{proof}
Write $\Omega=\Omega_B\times\Omega_{B^c}$ and
$\mu=\mu_B\otimes\mu_{B^c}$, where
$\Omega_B=\{-1,1\}^B$ and $\mu_B(\{\eta\})=2^{-|B|}$.
The function
\begin{align*}
 g(\xi)=2^{-|B|}\sum_{\eta\in\Omega_B}\mathbf1_A(\eta,\xi)
\end{align*}
is measurable on $\Omega_{B^c}$.
Note that every $C\in\mathcal G_B$ has the form $\Omega_B\times D$ for a
measurable set $D\subseteq\Omega_{B^c}$, and
\begin{align*}
 \int_C g\,d\mu
 =2^{-|B|}\sum_{\eta\in\Omega_B}
   \int_D\mathbf1_A(\eta,\xi)\,d\mu_{B^c}(\xi)
 =\mu(A\cap C).
\end{align*}
Thus, $g$ has the defining property of $\mu(A\mid\mathcal G_B)$.
\end{proof}

Enumerate the primes as $r_1<r_2<\cdots$, and put
\begin{align*}
 \mathcal F_n:=\sigma(\varepsilon_{r_i}:1\leq i\leq n),
\end{align*}
which is the $\sigma$-algebra generated by the events $\{\varepsilon_{r_i}=1\}$ for $1\leq i \leq n$.
Then we have 
\begin{align}\mathcal F_n\subseteq\mathcal F_{n+1},\qquad \mathcal F=\sigma\left(\bigcup_n\mathcal F_n\right).\nonumber\end{align}
The following lemma is an immediate consequence of L\'evy's
upward theorem; see \cite[ Theorem 4.6.8]{Durrett2019}.
\begin{lemma}\label{lem:conditionalgap}
For every $A\in\mathcal F$,
\begin{align}\label{eq:conditionalconvergence}
 \EE|\mu(A\mid\mathcal F_n)-\mathbf1_A|
 \longrightarrow0,\qquad n\rightarrow\infty.
\end{align}
Consequently, if there is $\delta>0$ such that
$\mu(A\mid\mathcal F_n)\leq1-\delta$ almost surely for every $n$,
then $\mu(A)=0$.
\end{lemma}

\subsection{Finite separation and non-concentration}
We prove that, for every fixed
$(c_1,\ldots,c_d)\in\CC^d\setminus\{0\}$,
the event $\sum_i c_iZ_i=0$ has probability zero.
The idea is to fix all signs outside a finite set $B$ of primes
and consider the finite products
\[
 \prod_{r\in B}
 (1-u_i(r)\varepsilon_r r^{-w}+r^{-2w})^{-1},
 \qquad 1\leq i\leq d,
\]
as functions of $(\varepsilon_r)_{r\in B}$.
The following linear-algebra lemma gives a sufficient condition
for these functions to be linearly independent. 
\begin{lemma}\label{lem:tensors}
For each pair $1\leq i<j\leq d$, choose a prime $r_{ij}$, with all
$\binom d2$ primes distinct, and put $B=\{r_{ij}:i<j\}$. For every
$r\in B$ and $1\leq i\leq d$, let
$v_{i,r}:\{-1,1\}\to\CC$ be a nonzero function. Suppose that
$v_{i,r_{ij}}$ and $v_{j,r_{ij}}$ are not proportional for every $i<j$.
Then the following functions on the vectors
$\eta=(\eta_r)_{r\in B}\in\{-1,1\}^{B}$ are linearly independent:
\begin{align*}
 F_i(\eta)=\prod_{r\in B}v_{i,r}(\eta_r)
 \qquad(1\leq i\leq d).
\end{align*}
\end{lemma}
\begin{proof}
For $d=1$, we have $B=\varnothing$ and hence $F_1\equiv1$,
so $\{F_1\}$ is linearly independent. Suppose $d\geq2$ and
$\sum_{k=1}^d a_kF_k=0$.
Fix $i$. We construct a linear functional $L_i$ such that 
$L_i(F_k)=0$ for $k\neq i$ and $L_i(F_i)\neq0$.

Write $r_{ji}=r_{ij}$. For $r=r_{ij}$ with $j\neq i$, define a
linear functional on the vector space of functions $h:\{-1,1\}\to\CC$ by
\begin{align*}
 \ell_r(h)=v_{j,r}(-1)h(1)-v_{j,r}(1)h(-1).
\end{align*}
It follows that $\ell_r(v_{j,r})=0$, and
$
 \ell_r(v_{i,r})
 =v_{j,r}(-1)v_{i,r}(1)-v_{j,r}(1)v_{i,r}(-1)\neq0$
since $v_{i,r}$ and $v_{j,r}$ are not proportional. At every remaining coordinate $r\in B$, choose
$e_r\in\{-1,1\}$ with $v_{i,r}(e_r)\neq0$ and set
$\ell_r(h)=h(e_r)$. Such a sign exists since $v_{i,r}$ is a nonzero
function. Because the primes $r_{ij}$ are distinct, $\ell_r$ is well-defined.  
For convenience, we write
$\ell_r(h)=b_r(1)h(1)+b_r(-1)h(-1)$. We define a linear functional
\begin{align*}
 L_i(H)=\sum_{\eta\in\{-1,1\}^{B}}
 \left(\prod_{r\in B}b_r(\eta_r)\right)H(\eta)
\end{align*}
for functions $H:\{-1,1\}^{B}\rightarrow\mathbb{C}$. It follows that
\begin{align*}
 L_i(F_k)
 &=\sum_{\eta\in\{-1,1\}^{B}}
       \prod_{r\in B}\bigl(b_r(\eta_r)v_{k,r}(\eta_r)\bigr)\\
 &=\prod_{r\in B}\left(\sum_{e\in\{-1,1\}}b_r(e)v_{k,r}(e)\right)
  \\&=\prod_{r\in B}\ell_r(v_{k,r}),
\end{align*}
which is nonzero when $k=i$ and zero otherwise.  Applying $L_i$ to the relation $\sum_{k=1}^d a_kF_k=0$ gives
$a_iL_i(F_i)=0$, hence $a_i=0$. Repeating this for each $i$
proves the lemma.
\end{proof}

\begin{proposition}\label{prop:non-concentration}
For every $\vect c=(c_1,\ldots,c_d)\in\CC^d\setminus\{0\}$,
\begin{align}\label{eq:randomzero}
\mu\left(\sum_{i=1}^dc_iZ_i=0\right)=0.
\end{align}
\end{proposition}
\begin{proof}
For $d=1$, the result follows from the non-vanishing in Lemma~\ref{lem:uniformproduct}.
Suppose $d\geq2$, fix $\vect c\neq0$, and put
$$A=\left\{\sum_i c_iZ_i=0\right\},\qquad M=\binom d2.$$
For each prime $r$, let $v_{i,r}(e)=V_{i,r}(e)$ for $e\in\{-1,1\}$, whose
 values are nonzero. Writing $t=r^{-w}$, we have by definition \eqref{eq:localV} that 
\begin{align*}
 \frac{V_{i,r}(1)}{V_{i,r}(-1)}
 =\frac{1+u_i(r)t+t^2}{1-u_i(r)t+t^2}.
\end{align*}
It follows that
\begin{align}\label{eq:localindependent}
 v_{i,r}\not\parallel v_{j,r}
 \quad&\Longleftrightarrow\quad 2t(1+t^2)(u_i(r)-u_j(r))\neq0\nonumber\\&\Longleftrightarrow\quad u_i(r)\neq u_j(r).
\end{align}

For fixed $n$ and each pair $i<j$, choose a prime $r_{ij}>r_n$ with $u_i(r_{ij})\neq u_j(r_{ij})$, making all choices distinct. This is possible by the assumption \eqref{eq:abstractseparation}, and at each choice we only exclude the first $n$ primes and the primes previously chosen.   Put
\begin{align*}
 B_n=\{r_{ij}:i<j\},\qquad
 \mathcal G_n=\sigma(\varepsilon_r:r\notin B_n).
\end{align*}
Then $|B_n|=M$, and Lemma~\ref{lem:tensors} implies that 
\begin{align*}
 F_i(\eta)=\prod_{r\in B_n}V_{i,r}(\eta_r)
 \qquad(\eta\in\{-1,1\}^{B_n}, 1\leq i\leq d),
\end{align*}
are linearly independent. Note that for any $R\geq\max B_n$, we have 
\begin{align*}
    \prod_{r\leq R}V_{i,r}(\varepsilon_r)=\left(\prod_{r\in B_n}V_{i,r}(\varepsilon_r)\right)\prod_{r\leq R,r\notin B_n}V_{i,r}(\varepsilon_r).
\end{align*}
Then Lemma \ref{lem:uniformproduct} gives
\begin{align*}
    Z_i=D_iF_i\left((\varepsilon_r)_{r\in B_n}\right),\qquad D_i:=\prod_{r\notin B_n}V_{i,r}(\varepsilon_r), \qquad (1\leq i\leq d),
\end{align*}
where $D_i$'s are nonzero. Hence $c_i D_i$'s are not all zero, and the linear independence of $F_i$'s implies that $\sum_{i=1}^dc_i D_iF_i(\eta)$ is nonzero for at least one $\eta$. It follows from Lemma~\ref{lem:blockconditioning} that 
\begin{align*}\mu(A\mid \mathcal{G}_n)&=2^{-|B_n|}\sum_{\eta\in\{-1,1\}^{B_n}}\mathbf{1}_A(\omega^{B_n,\eta})\\&\leq 2^{-M}(2^M-1)\\&=1-2^{-M}\qquad{\rm almost~surely}.\end{align*}

Finally, note that $\mathcal{F}_n\subseteq\mathcal{G}_n$, then the tower property \eqref{eq:tower} gives 
\begin{align*}
    \mu(A\mid \mathcal{F}_n)&= \EE(\mu(A\mid\mathcal G_n)\mid\mathcal F_n)\\&\leq 1-2^{-M}\qquad{\rm almost~surely}.
\end{align*}
Therefore, we have $\mu(A)=0$ from Lemma \ref{lem:conditionalgap}, completing the proof.
\end{proof}

\subsection{From the random signs to prime quadratic twists}\label{subsect:passtotwist} We show \eqref{eq:twistP=0} and finish the proof of Theorem \ref{thm:abstract}. Fix $b\in\{1,3\}$. Recall that $\cP_{4,b}(X)$ is the set of primes
$p\leq X$ with $p\equiv b\pmod4$, and $P_b(X)=\#\mathscr{P}_{4,b}(X)$.
For $P_b(X)>0$, we define the probability space $(\Omega_X,\mathcal{F}_X,\mu_X)$ by:
\begin{align}
    \Omega_X=\mathscr{P}_{4,b}(X),\qquad\mathcal{F}_X=2^{\Omega_X},\qquad \mu_X(A)=\frac{\#A}{P_b(X)},
\end{align}
where $2^{\Omega_X}$ denotes the collection of all subsets of $\mathcal{P}_{4,b}(X)$. 
For an odd prime $p$, we define 
\begin{align*}
 Y(p)=(Y_1(p),\ldots,Y_d(p)),
 \qquad Y_i(p):=\prod_r V_{i,r}(\chi_p(r)).
\end{align*}
Then $Y$ is a random vector on $(\Omega_X,\mathcal{F}_X,\mu_X)$ and its distribution is given by
\begin{align}
 \nu_X(E):=
 \frac{\#\{p\in\cP_{4,b}(X):Y(p)\in E\}}{P_b(X)}
 \qquad(E\subseteq\CC^d\text{ Borel}).\label{eq:distributionofY(p)}
\end{align}
Since $Y(p)\in K$ for every $p\in\Omega_X$, we have $\nu_X(K)=1$.
We will show that, as $p$ varies, the distribution of $Y(p)$ approaches
that of $Z$, and then transfer Proposition~\ref{prop:non-concentration}
to the prime twists.

For a nonzero coefficient vector $\vect c=(c_1,\ldots,c_d)$, write
\begin{align*}
 H_{\vect c}=\{z\in\CC^d:\vect c\cdot z=0\},
 \qquad \vect c\cdot z=\sum_{i=1}^d c_i z_i,
\end{align*}
which is a complex linear hyperplane through the origin. Note that 
$$\nu_X(H_{\vect c})=\frac{
   \#\{p\in\cP_{4,b}(X):\vect c\cdot Y(p)=0\}}{P_b(X)}.$$ 
   Our goal is to prove
$$\sup_{\vect c\neq0}\nu_X(H_{\vect c})\to0,\qquad X\to\infty.$$
Equivalently, for any $\epsilon>0$, there exists some $X_0$ independent of $\vect{c}$ such that
$\nu_X(H_{\vect{c}})<\epsilon,$
for $X\geq X_0$ and $\vect{c}\neq \vect{0}$.

For probability measures $\nu_n$ and $\nu$
on $\CC^d$, we use the notation $\nu_n\Longrightarrow\nu$ to mean
\emph{weak convergence}:
\begin{align*}
\int\varphi\,d\nu_n\longrightarrow\int\varphi\,d\nu
 \quad\text{for every bounded continuous }\varphi:\CC^d\to\mathbb R.
\end{align*}
The following consequence of \cite[Theorem~4(iii), p.~3]{uniformweakconvergence1967} gives uniform convergence over
complex linear hyperplanes through the origin.

\begin{lemma}\label{lem:compactness}
Let $\nu_n$ and $\nu$ be probability measures supported on a common
compact set $K\subseteq\CC^d$. Suppose $\nu_n\Longrightarrow\nu$
and $\nu(H)=0$ for every complex linear hyperplane $H$ through the
origin. Then
\begin{align*}
 \sup_H\nu_n(H)\longrightarrow0,
\end{align*}
where the supremum is over these hyperplanes.
\end{lemma}

We are now ready to prove \eqref{eq:twistP=0}.
\begin{proposition}\label{prop:primeproducts}
Under \eqref{eq:abstractseparation}, for each $b\in\{1,3\}$,
\begin{align}\label{eq:abstractuniform}
 \sup_{\vect c\neq0}
 \#\{p\in\cP_{4,b}(X):\vect c\cdot Y(p)=0\}=o(P_b(X)).
\end{align}
\end{proposition}

\begin{proof}
We first show that $\nu_X\Longrightarrow\nu$.
Fix a bounded continuous function $\varphi:\CC^d\to\mathbb R$.
For $R\geq2$, define the truncated coordinates and random vectors by
\begin{align*}
 Y_{i,R}(p)&=\prod_{r\leq R}V_{i,r}(\chi_p(r)),
 &Z_{i,R}&=\prod_{r\leq R}V_{i,r}(\varepsilon_r),\\
 Y_R(p)&=(Y_{1,R}(p),\ldots,Y_{d,R}(p)),
 &Z_R&=(Z_{1,R},\ldots,Z_{d,R}),
\end{align*}
where the products are over primes. The values of $Y_R$ and $Z_R$ lie in the compact set $K$ defined in \eqref{eq:compactK}.
Fix $R$. Let $r_1<\cdots<r_t$ be the primes at most $R$.
For $\eta=(\eta_1,\ldots,\eta_t)\in\{-1,1\}^t$,
put
\begin{align*}
 y_R(\eta)&=
 \left(\prod_{j=1}^tV_{1,r_j}(\eta_j),\ldots,
       \prod_{j=1}^tV_{d,r_j}(\eta_j)\right),\\
 N_X(\eta)&=
 \#\{p\in\cP_{4,b}(X):p>R,\
       \chi_p(r_j)=\eta_j\text{ for }1\leq j\leq t\}.
\end{align*}

Note that for $p>R$, we have  $\chi_p(r_j)\neq0$, $\eta=(\chi_{p}(r_1),\ldots,\chi_p(r_{t}))\in\{-1,1\}^t$, and $Y_R(p)=y_R(\eta)$. It follows that
\begin{align}\label{eq:midstep}
 \frac1{P_b(X)}
 \sum_{\substack{p\in\cP_{4,b}(X)\\p>R}}\varphi(Y_R(p))
 &=\frac{1}{P_b(X)}\sum_{\eta\in\{-1,1\}^t}\sum_{\substack{p\in\cP_{4,b}(X), ~p>R\\\chi_p(r_j)=\eta_j~(1\leq j\leq t)}}\varphi(y_R(\eta))\nonumber\\&=
 \sum_{\eta\in\{-1,1\}^t}
   \frac{N_X(\eta)}{P_b(X)}\,\varphi(y_R(\eta))\nonumber\\&\longrightarrow 2^{-t}\sum_{\eta\in\{-1,1\}^t}\varphi(y_R(\eta)),\qquad X\rightarrow\infty,
\end{align}
where we used the fact $N_X(\eta)/P_b(X)\rightarrow 2^{-t}$ by Lemma~\ref{lem:signs}. Note also that
\begin{align}
 \left|
 \frac1{P_b(X)}
 \sum_{\substack{p\in\cP_{4,b}(X)\\p\leq R}}
 \varphi(Y_R(p))
 \right|
 \leq \frac{t}{P_b(X)}\sup_{z\in K}|\varphi(z)|
 \longrightarrow0,\qquad X\rightarrow\infty. \label{eq:midstepp<R}
\end{align}
Here $R$ and $t$ are fixed, $\varphi$ is bounded on the compact set $K$,
and $P_b(X)\to\infty$. 

On the other hand, for each $\eta=(\eta_1,\ldots,\eta_t)\in\{-1,1\}^t$, let $E_{\eta}$ denote the event $\{\varepsilon_{r_j}=\eta_j:1\leq j\leq t\}$.
Then the independence of $(\varepsilon_r)_r$ gives
\begin{align*}
\mu(E_{\eta})
=\prod_{j=1}^t\mu(\varepsilon_{r_j}=\eta_j)
 =2^{-t}.
\end{align*}
Note that we have  $Z_R=y_R(\eta)$ on $E_{\eta}$. Since $\Omega=\bigsqcup_{\eta\in\{-1,1\}^t} E_{\eta}$, we have
\begin{align}
 \EE\varphi(Z_R)
 =&\sum_{\eta\in\{-1,1\}^t}\int_{E_{\eta}}\varphi(Z_R)d\mu\nonumber\\=&\sum_{\eta\in\{-1,1\}^t}
   \varphi(y_R(\eta))\,
\mu(E_{\eta})\nonumber\\
 =&2^{-t}\sum_{\eta\in\{-1,1\}^t}\varphi(y_R(\eta)).\label{eq:midstepZ_R}
\end{align}
Thus, by \eqref{eq:midstep}, \eqref{eq:midstepp<R}, and \eqref{eq:midstepZ_R}, we have
\begin{align}\label{eq:truncatedlimit}
\lim_{X\to\infty}\frac1{P_b(X)}
 \sum_{p\in\cP_{4,b}(X)}\varphi(Y_R(p))
 =2^{-t}\sum_{\eta\in\{-1,1\}^t}\varphi(y_R(\eta))
 =\EE\varphi(Z_R).
\end{align}

We now remove the truncation. Lemma~\ref{lem:uniformproduct} gives
$Y_R(p)\to Y(p)$ uniformly in $p$ and $Z_R(\omega)\to Z(\omega)$
uniformly in $\omega\in\Omega$. Since all these random vectors lie in $K$,
uniform continuity of $\varphi$ on $K$ gives a common bound
$\epsilon_R\to0$ such that
\begin{align}
 |\varphi(Y(p))-\varphi(Y_R(p))|\leq\epsilon_R,\qquad
 |\varphi(Z(\omega))-\varphi(Z_R(\omega))|\leq\epsilon_R\label{eq:midstepuniformbd}
\end{align}
for every odd prime $p$ and every $\omega\in\Omega$.
Hence \eqref{eq:truncatedlimit} and \eqref{eq:midstepuniformbd}  give
\begin{align*}
 \limsup_{X\to\infty}
 \left|\int\varphi\,d\nu_X-\int\varphi\,d\nu\right|
 \leq2\epsilon_R.
\end{align*}
Letting $R\to\infty$ proves the weak convergence $\nu_X\Longrightarrow\nu$.

Finally, note that
Proposition~\ref{prop:non-concentration} gives
\begin{align*}
 \nu(H_{\vect c})=\mu\left(\sum_{i=1}^d c_iZ_i=0\right)=0
 \qquad(\vect c\neq0).
\end{align*}
As both $\nu$ and $\nu_X$ are supported on a compact set $K$, we can apply
Lemma \ref{lem:compactness},  along any sequence $X\to\infty$,
giving
\begin{align*}
 \sup_{\vect c\neq0}\nu_X(H_{\vect c})
 =\frac1{P_b(X)}\sup_{\vect c\neq0}
   \#\{p\in\cP_{4,b}(X):\vect c\cdot Y(p)=0\}
 \longrightarrow0,
\end{align*}
which proves \eqref{eq:abstractuniform}.
\end{proof}

Therefore, Propositions~\ref{prop:non-concentration} and \ref{prop:primeproducts}
prove 
Theorem~\ref{thm:abstract}.

\section{Proofs of the main results}\label{sect:proofs}

\begin{proof}[Proof of Theorem \ref{thm:standard}]
Put $w=s-(k-1)/2$ and $u_i(r)=\lambda_{f_i}(r)$.
Then $\re(w)>1$, $u_i(r)\in[-2,2]$, and
Lemma~\ref{lem:separation} gives the condition \eqref{eq:abstractseparation}.
By \eqref{eq:standardproduct}, the products in
Proposition~\ref{prop:primeproducts} are $Y_i(p)=L(f_i,\chi_p,s)$.
For $h=\sum_{i=1}^d c_if_i$, the absolute convergence of $Y_i(p)$ gives
\begin{align*}
 L(h,\chi_p,s)=\sum_{i=1}^d c_iY_i(p).
\end{align*}
Since $h\neq0$ exactly when $\vect c\neq0$, we have 
$$\sup_{0\neq h\in S_k}
 \#\{p\in\cP_{4,b}(X):L(h,\chi_p,s)=0\}= \sup_{\vect c\neq0}
 \#\{p\in\cP_{4,b}(X):\vect c\cdot Y(p)=0\}.$$
Thus, the result follows from
\eqref{eq:abstractuniform}.
\end{proof}

\begin{proof}[Proof of Theorem~\ref{thm:symmetric}]
Put $w=s-k+1$ and $u_i(r)=b_{f_i}(r)$.
Then again we have $\re(w)>1$, $u_i(r)\in[-2,2]$ by
\eqref{eq:normalizedcoeff}, and Lemma~\ref{lem:separation}
gives \eqref{eq:abstractseparation}. For the corresponding
products $Y_i(p)$ in Proposition~\ref{prop:primeproducts},
\eqref{eq:reducedsym} gives
\begin{align*}
 \sum_{i=1}^d c_iL(\Sym^2f_i,\chi_p,s)
 =L(w,\chi_p)\sum_{i=1}^d c_iY_i(p),
\end{align*}
where $L(w,\chi_p)\neq0$. Thus, $\sum_{i=1}^d c_iL(\Sym^2f_i,\chi_p,s)$ vanishes exactly
when $\sum_{i=1}^d c_iY_i(p)=0$. Taking the supremum over
$\vect c\neq0$ and applying \eqref{eq:abstractuniform} proves
\eqref{eq:symuniform}.
\end{proof}
\begin{proof}[Proof of Corollary~\ref{cor:determinant}]
Fix $b\in\{1,3\}$, $m\in\{1,2\}$, and $s\in\CC$ with
$\re(s)>\sigma_m$. Put
\begin{align*} v(p)&=(L^{(m)}(f_1,\chi_p,s),\ldots,L^{(m)}(f_d,\chi_p,s))^T
,\\
 B_X&=\sup_{\vect c\neq0}
 \#\{p\in\cP_{4,b}(X):\vect c\cdot v(p)=0\},\\P&=P_b(X).\end{align*}
By \eqref{eq:standarduniform} and \eqref{eq:symuniform},
$B_X=o(P)$. 
Note that every proper subspace $W\subsetneq\CC^d$ is contained in
$\{z:\vect c\cdot z=0\}$ for some $\vect c\neq0$. Hence
\[
 \#\{p\in\cP_{4,b}(X):v(p)\in W\}\le B_X
 \qquad(W\subsetneq\CC^d).
\]
Since each coordinate of $v(p)$ is a nonzero Euler product,  the
case $d=1$ is immediate. Now we assume $d\ge2$. If the matrix with
columns $v(p_1),\ldots,v(p_d)$ is singular, then there is a smallest
 index $j\in\{2,\ldots,d\}$ such that 
$v(p_1),\ldots,v(p_{j-1})$ are linearly independent and 
 $v(p_j)\in
 W:=\Span_{\CC}\{v(p_1),\ldots,v(p_{j-1})\}.$
Fix this index $j$. There are at most $P^{j-1}$ ordered choices
of the preceding primes $p_1,\ldots,p_{j-1}$ whose corresponding columns  
$v(p_1),\ldots,v(p_{j-1})$ are linearly independent. For each such choice, $\dim W=j-1<d$, so at most
$B_X$ primes can be chosen for $p_j$. This bound is independent
of the preceding choices. The remaining primes
$p_{j+1},\ldots,p_d$ can be chosen in $P^{d-j}$ ways.
Therefore, the number of  tuples $(p_1,\ldots,p_d)$ whose first dependence
occurs at $j$ is at most
$$
 P^{j-1}B_XP^{d-j}=B_XP^{d-1}.$$
Summing over the possible indices of $j$ gives
\[
 \#\{(p_1,\ldots,p_d)\in\cP_{4,b}(X)^d:
       \det(v(p_1),\ldots,v(p_d))=0\}
 \le (d-1)B_XP^{d-1}=o(P^d).
\]

Finally, let $T$ be a fixed finite collection of pairs $(m,s)$ with $m\in\{1,2\}$ and $\re(s)>\sigma_m$. For each $(m,s)\in T$, let 
$$E_{m,s}(X)=\left\{(p_1,\ldots,p_d)\in\cP_{4,b}(X)^d:
 \det(L^{(m)}(f_i,\chi_{p_j},s))_{1\leq i,j\leq d}=0\right\}.$$
The preceding argument gives
$\#E_{m,s}(X)=o(P^d)$. Since $T$ is fixed and finite,
\[
 \#\bigcup_{(m,s)\in T}E_{m,s}(X)
 \le \sum_{(m,s)\in T}\#E_{m,s}(X)
 =o(P^d).
\]
Outside the above union, the determinants are nonzero for every
$(m,s)\in T$. Thus, the proof is complete.
\end{proof}

To prove the Theorems \ref{thm:standard}, \ref{thm:symmetric}, and Corollary \ref{cor:determinant} for the odd primes version as mentioned in Remark~\ref{rem:scope}, it suffices to notice that $\mathcal{P}_{\odd}(X)=\mathcal{P}_{4,1}(X)\sqcup\mathcal{P}_{4,3}(X)$ and $\#\mathcal{P}_{\odd}(X)=P_1(X)+P_3(X)$, and repeat the preceding counting proof. We omit the proofs.  

\section{Kernel functions and spanning subsets}\label{sect:kernels}
Kohama and Mizuno \cite{KohamaMizuno2018,KohamaMizuno2019}
studied kernels for twisted symmetric square $L$-functions and applied
them to the computation of exact special values. Here, we recall the standard kernel representations.

For $h,g\in S_k$, the Petersson inner product is
\begin{align*}
 \langle h,g\rangle
 =\int_{\Sl_2(\ZZ)\backslash\mathbb H}
 h(z)\overline{g(z)}y^k\,\frac{dx\,dy}{y^2},
 \qquad z=x+iy.
\end{align*}
Write $\|h\|^2=\langle h,h\rangle$ and $N_i=\|f_i\|^2$.
For $m\in\{1,2\}$ and $\re(s)>\sigma_m$, put
\begin{align*}
 \mathcal L^{(m)}_{p,s}\left(\sum_{i=1}^d c_if_i\right)
 =\sum_{i=1}^d c_iL^{(m)}(f_i,\chi_p,s).
\end{align*}
Thus $\mathcal L^{(1)}_{p,s}(h)=L(h,\chi_p,s)$, and
$\mathcal L^{(2)}_{p,s}$ is the linear functional defined in
\eqref{eq:symfunctional}.

Let $\Gamma=\Sl_2(\ZZ)$ and
$\Gamma_\infty=\{\pm\left(\begin{smallmatrix}1&n\\0&1\end{smallmatrix}\right):n\in\ZZ\}$.
For $n\geq1$, the  Poincar\'e series and its normalization are
\begin{align}\label{eq:poincare}
 P_{k,n}(z)&=\sum_{\gamma\in\Gamma_\infty\backslash\Gamma}
 \frac{e^{2\pi i n\gamma z}}{j(\gamma,z)^k},&
 Q_{k,n}(z)&=\frac{(4\pi n)^{k-1}}{\Gamma(k-1)}P_{k,n}(z),
\end{align}
where $j(\gamma,z)=cz+d$ for
$\gamma=\left(\begin{smallmatrix}a&b\\c&d\end{smallmatrix}\right)$.
These forms belong to $S_k$, and the Petersson coefficient formula
\cite[Theorem~8.2.3]{CohenStromberg2017} gives
\begin{align}\label{eq:poincarepairing}
 \langle h,Q_{k,n}\rangle=a_h(n),\qquad
 Q_{k,n}=\sum_{i=1}^d\frac{\overline{a_{f_i}(n)}}{N_i}f_i.
\end{align}
For the symmetric square normalization, write
\begin{align}\label{eq:squarefactor}
 D_p(s)=L(2s-2k+2,\chi_p^2)
 =\zeta(2s-2k+2)\bigl(1-p^{2k-2-2s}\bigr).
\end{align}
We have the Dirichlet-series identity
\begin{align}\label{eq:squaredirichlet}
 L(\Sym^2f,\chi_p,s)
 =D_p(s)\sum_{n\geq1}\frac{a_f(n^2)\chi_p(n)}{n^s},
 \qquad f\in\HH_k,\quad\re(s)>k.
\end{align}
Recall the following standard kernel formulas; see also \cite[pp.~67-68]{Stopplesymmetricsquare}.
\begin{proposition}\label{thm:kernelconstruction}
For an odd prime $p$, the series
\begin{align}
 K^{(1)}_{k,p,s}(z)
 &=\sum_{n\geq1}\frac{\chi_p(n)}{n^{\overline{s}}}Q_{k,n}(z),
 &&\re(s)>(k+1)/2,\label{eq:standardkernel}\\
 K^{(2)}_{k,p,s}(z)
 &=\overline{D_p(s)}\sum_{n\geq1}
 \frac{\chi_p(n)}{n^{\overline{s}}}Q_{k,n^2}(z),
 &&\re(s)>k\label{eq:symkernel}
\end{align}
converge absolutely in Petersson norm and locally uniformly on
$\mathbb H$. For $m\in\{1,2\}$ and $\re(s)>\sigma_m$, the resulting
form is the unique element of $S_k$ satisfying
\begin{align}\label{eq:kernelpairing}
 \langle h,K^{(m)}_{k,p,s}\rangle=\mathcal L^{(m)}_{p,s}(h)
 \qquad(h\in S_k),
\end{align}
and its spectral expansion is
\begin{align}\label{eq:kernelspectral}
 K^{(m)}_{k,p,s}
 =\sum_{i=1}^d\frac{\overline{L^{(m)}(f_i,\chi_p,s)}}{N_i}f_i.
\end{align}
\end{proposition}

We next introduce the density terminology used to describe subsets of
these kernels. 

\begin{definition}[Relative densities]\label{def:relativedensity}
For $\mathcal A\subseteq\cP_{4,b}$, put
$A(X)=\#(\mathcal A\cap[1,X])$. Its upper and lower relative densities are
\begin{align*}
 \overline\delta_b(\mathcal A)
 &=\limsup_{X\to\infty}\frac{A(X)}{P_b(X)},&
 \underline\delta_b(\mathcal A)
 &=\liminf_{X\to\infty}\frac{A(X)}{P_b(X)}.
\end{align*}
If these quantities agree, their common value is the relative density
of $\mathcal A$ in $\cP_{4,b}$.
\end{definition}

\begin{theorem}\label{thm:kernelspanning}
Fix $b\in\{1,3\}$, $m\in\{1,2\}$, and $s$ with
$\re(s)>\sigma_m$. Then
\begin{align}\label{eq:kerneluniform}
 B_{m,s,b}(X):=\sup_{W\subsetneq S_k}
 \#\{p\in\cP_{4,b}(X):K^{(m)}_{k,p,s}\in W\}
 =o(P_b(X)),
\end{align}
where $W$ runs over proper complex linear subspaces of $S_k$.
Consequently:
\begin{enumerate}
 \item For every $\delta>0$, there exists $X_0$ such that, for every
 $X\geq X_0$ and every $\mathcal A_X\subseteq\cP_{4,b}(X)$ with
 $\#\mathcal A_X\geq\delta P_b(X)$,
 \begin{align*}
  \Span_{\CC}\{K^{(m)}_{k,p,s}:p\in\mathcal A_X\}=S_k.
 \end{align*}
 \item If $\mathcal A\subseteq\cP_{4,b}$ satisfies
 $\overline\delta_b(\mathcal A)>0$, then, for every $Y>0$,
 \begin{align*}
  \Span_{\CC}\{K^{(m)}_{k,p,s}:p\in\mathcal A,\ p>Y\}=S_k.
 \end{align*}
 \item Almost every ordered tuple $(p_1,\ldots,p_d)\in\cP_{4,b}(X)^d$
 gives a basis $K^{(m)}_{k,p_1,s},\ldots,K^{(m)}_{k,p_d,s}$ of $S_k$.
\end{enumerate}
The estimate \eqref{eq:kerneluniform} and assertions (i)--(iii) hold
simultaneously for any fixed finite collection of pairs $(m,s)$ with
$m\in\{1,2\}$ and $\re(s)>\sigma_m$.
\end{theorem}
\begin{proof}
Note that every proper complex subspace $W\subsetneq S_k$ is contained in
$h^\perp=\{g\in S_k:\langle h,g\rangle=0\}$ for some $h\neq0$.
By \eqref{eq:kernelpairing}, $K^{(m)}_{k,p,s}$ belongs to
$h^\perp$ exactly when $\mathcal L^{(m)}_{p,s}(h)=0$.
Since $h^\perp$ is itself a proper subspace, we obtain
\begin{align*}
 B_{m,s,b}(X)
 =\sup_{0\neq h\in S_k}
 \#\{p\in\cP_{4,b}(X):\mathcal L^{(m)}_{p,s}(h)=0\}.
\end{align*}
Thus, Theorem~\ref{thm:standard} gives \eqref{eq:kerneluniform} for $m=1$,
and Theorem~\ref{thm:symmetric} gives \eqref{eq:kerneluniform} for $m=2$.

For (i), fix $\delta>0$ and choose $X_0$ so that
$B_{m,s,b}(X)<\delta P_b(X)$ for every $X\geq X_0$.
For such an $X$, set
$W_X=\Span_{\CC}\{K^{(m)}_{k,p,s}:p\in\mathcal A_X\}$.
If $W_X\subsetneq S_k$, the definition of $B_{m,s,b}(X)$ gives the contradiction
\begin{align*}
 \delta P_b(X)\leq\#\mathcal A_X\leq B_{m,s,b}(X)<\delta P_b(X).
\end{align*}

For (ii), fix $Y>0$ and put $\mathcal A_{>Y}=\{p\in\mathcal A:p>Y\}$.
Removing finitely many primes gives
$\overline\delta_b(\mathcal A_{>Y})=\overline\delta_b(\mathcal A)>0$.
If $W_Y=\Span_{\CC}\{K^{(m)}_{k,p,s}:p\in\mathcal A_{>Y}\}\subsetneq S_k$, then
\begin{align*}
 \#(\mathcal A_{>Y}\cap[1,X])\leq B_{m,s,b}(X)=o(P_b(X)),
\end{align*}
which implies that $\overline\delta_b(\mathcal A_{>Y})=0$, 
leading to a contradiction.

For (iii), fix $\vect p=(p_1,\ldots,p_d)\in\cP_{4,b}(X)^d$ and write
\begin{align*}
 M_{m,s}(\vect p)
 &=(L^{(m)}(f_i,\chi_{p_j},s))_{1\leq i,j\leq d},&
 D&=\operatorname{diag}(N_1,\ldots,N_d).
\end{align*}
By \eqref{eq:kernelspectral}, the $j$th column of
$D^{-1}\overline{M_{m,s}(\vect p)}$ consists of the coefficients of
$K^{(m)}_{k,p_j,s}$ in the basis $f_1,\ldots,f_d$.
Since every $N_i>0$, the kernels indexed by $\vect p$ form a basis
if and only if $\det M_{m,s}(\vect p)\neq0$.
Corollary~\ref{cor:determinant} bounds the number of tuples with
$\det M_{m,s}(\vect p)=0$ by $o(P_b(X)^d)$, proving (iii).

For a fixed nonempty finite collection $\mathscr T$ of pairs $(m,s)$ with $m\in\{1,2\}$ and $\re(s)>\sigma_m$, put
\begin{align*}
 B_{\mathscr T}(X)
 =\max_{(m,s)\in\mathscr T}B_{m,s,b}(X)=o(P_b(X)).
\end{align*}
Using $B_{\mathscr T}(X)$ in the proofs of (i) and (ii) proves the
spanning conclusions for every $(m,s)\in\mathscr T$.
For (iii), Corollary~\ref{cor:determinant} bounds by $o(P_b(X)^d)$
the number of tuples for which $\det M_{m,s}(\vect p)=0$ for at
least one $(m,s)\in\mathscr T$. This completes the proof.
\end{proof}

\section{Applications to periods and rational bases}\label{sect:applications}

We now apply the uniform non-vanishing theorems to twisted 
periods and rational cusp forms obtained from  Eisenstein series at prime levels.

\subsection{Twisted periods}
For $0\leq t\leq k-2$ and $h(z)=\sum_{n\geq1}a_h(n)q^n\in S_k$, put
\begin{align}
 r_{t,\chi_p}(h)
 =\int_0^{i\infty}(h\otimes\chi_p)(z)z^t\,dz
 =\frac{t!}{(-2\pi i)^{t+1}}L(h,\chi_p,t+1),\label{eq:defoftwistperiod}
\end{align}
where
\begin{align*}
 (h\otimes\chi_p)(z)=\sum_{n\geq1}a_h(n)\chi_p(n)q^n
 \in S_k(p^2,\chi_p^2);
\end{align*}
see \cite[Proposition~14.19]{IwaniecKowalski2004}. Here
$L(h,\chi_p,s)$ is the twisted $L$-function defined in the introduction.

Note that each $r_{t,\chi_p}$ is a linear functional on $S_k$.
For each fixed noncentral index $t$ ($t\neq k/2-1$), these functionals,
with $p\in\cP_{4,b}$, span $S_k^*$ by
\cite[Theorem~5.1]{Ni2026}. It remains open whether every choice
of $d$ distinct primes in $\cP_{4,b}$ gives a basis. The following
corollary shows that almost every tuple gives a basis simultaneously
at all such indices.

\begin{corollary}\label{cor:periods}
Fix $b\in\{1,3\}$. Outside a set of $o(P_b(X)^d)$ tuples
$(p_1,\ldots,p_d)\in\cP_{4,b}(X)^d$, the periods
\begin{align*}
r_{t,\chi_{p_1}},\ldots,r_{t,\chi_{p_d}}
\end{align*}
form a basis of $S_k^*$ simultaneously for every integer
\begin{align*}
 0\leq t\leq k/2-2\qquad\text{or}\qquad k/2\leq t\leq k-2.
\end{align*}
\end{corollary}
\begin{proof}
Since $f_1,\ldots,f_d$ form a basis of $S_k$, the periods form a
basis of $S_k^*$ exactly when the matrix
$R_t=(r_{t,\chi_{p_j}}(f_i))_{i,j=1}^d$ is nonsingular.
If $k/2\leq t\leq k-2$, then
\begin{align*}
 \det R_t
 =\left(\frac{t!}{(-2\pi i)^{t+1}}\right)^d
\det(L(f_i,\chi_{p_j},t+1))_{1\leq i,j\leq d}.
\end{align*}
Here $t+1\geq k/2+1>(k+1)/2$, so
Corollary~\ref{cor:determinant} applies.

Now suppose $0\leq t\leq k/2-2$. For $f\in\HH_k$, put
\begin{align*}
 \Lambda_{f,p}(s)
 =\left(\frac p{2\pi}\right)^s\Gamma(s)L(f,\chi_p,s).
\end{align*}
The quadratic-twist functional equation
\cite[Chapter~14]{IwaniecKowalski2004}, in the normalization recorded
in \cite[Section~2.1]{Ni2026}, is
\begin{align*}
 \Lambda_{f,p}(s)=i^k\chi_p(-1)\Lambda_{f,p}(k-s).
\end{align*}
Consequently,
\begin{align*}
 L(f,\chi_p,s)
 &=i^k\chi_p(-1)\left(\frac p{2\pi}\right)^{k-2s}
 \frac{\Gamma(k-s)}{\Gamma(s)}L(f,\chi_p,k-s).
\end{align*}
Setting $s=t+1$ and using \eqref{eq:defoftwistperiod} gives
\begin{align*}
 r_{t,\chi_{p_j}}(f_i)
 &=C_{p_j,t}L(f_i,\chi_{p_j},k-t-1),\\
 C_{p,t}
 &:=\frac{i^k\chi_p(-1)(k-t-2)!}{(-2\pi i)^{t+1}}
 \left(\frac p{2\pi}\right)^{k-2t-2},
\end{align*}
where the factor $C_{p,t}$ is nonzero and independent of $f_i$. Hence
\begin{align*}
 \det R_t
 =\left(\prod_{j=1}^d C_{p_j,t}\right)
 \det(L(f_i,\chi_{p_j},k-t-1))_{i,j=1}^d.
\end{align*}
Since $k-t-1\geq k/2+1>(k+1)/2$,
Corollary~\ref{cor:determinant} also applies in this range.
Finally, since there are only finitely many admissible indices $t$, the union
of their exceptional sets still has size $o(P_b(X)^d)$.
\end{proof}

\begin{remark}
    One may compare Corollary \ref{cor:periods} to \cite[Theorems 1.1-1.4]{NiXueTwisted2026}, in which the authors study the independence when the weight $k$ of the space $S_k$ is sufficiently large relative to the conductor and the number of twisted periods. Here, the weight is fixed, and prime conductors vary, producing density-one conclusions. 
\end{remark}

\subsection{Rational bases from Eisenstein series at prime levels}
We recall the construction in \cite[Definitions 1.1--1.3]{Ni2026}. Let $k_1,k_2\geq3$ be
integers of the same parity with $k_1<k_2$, and let $c\geq1$ be an integer. For
$p\equiv(-1)^{k_1}\pmod4$ and $a\in\{k_1,k_2\}$, we have the Eisenstein series
\begin{align*}
 G_{a,\chi_p}(z)
 =\frac12L(1-a,\chi_p)
   +\sum_{n\geq1}\left(\sum_{v\mid n}\chi_p(v)v^{a-1}\right)q^n.
\end{align*}

Let $f(z)\in M_{a}(N,\nu_1)$ and $g(z)\in M_{b}(N,\nu_2)$ be modular forms of level $N$, weights $a$ and $b$ and nebentypus $\nu_1$ and $\nu_2$, respectively. For an integer $e\geq1$, the $e$-th \textit{Rankin-Cohen bracket} is given by
\begin{align*}
    [f(z),g(z)]_e := \sum_{r=0}^e (-1)^r\binom{e+a-1}{e-r}\binom{e+b-1}{r}f(z)^{(r)}g(z)^{(e-r)},
\end{align*}
where $f(z)^{(r)}$ is the $r$-th normalized derivative $f(z)^{(r)}:=\frac{1}{(2\pi i)^r}\frac{d^r f(z)}{dz^r}$ of $f$.  Moreover, $[f,g]_e\in S_{a+b+2e}(N,\nu_1\nu_2)$; see \cite[Theorem 7.1]{Cohen'smodularformC_k}. 

 We define 
\begin{align}
F_p=\Tr_1^p[G_{k_1,\chi_p},G_{k_2,\chi_p}]_c\in S_K,
 \qquad K=k_1+k_2+2c.
\end{align}
where $\Tr_1^p$  is the trace map 
\begin{align*} 
    \Tr_1^p: M_{m}(p)\rightarrow M_m(1),\quad g\mapsto \sum_{\gamma\in\Gamma_0(p)\backslash\Gamma_0(1)}g|_m\gamma,
\end{align*}
and for $m\in\mathbb{Z}$ and $\gamma=\begin{bsmallmatrix}
    a&b\\c&d
\end{bsmallmatrix}\in \Gl_2^+(\mathbb{R})$  the slash operator is
\begin{align*} 
g(z)|_m\gamma=\det(\gamma)^{m/2} (cz+d)^{-m} g\left(\frac{az+b}{cz+d}\right).
\end{align*}
 In particular, $F_p$ has rational Fourier coefficients; see \cite[Proposition 4.3]{rampantjantauni2026} and \cite[(1.3)]{Ni2026}. Let
$d_K=\dim S_K$, and assume $d_K\geq1$.
With $k_1,k_2,c$ fixed, the following result proves the density-one
version of \cite[Conjecture 1.4]{Ni2026}, which asks that every
$d_K$ distinct admissible primes give a basis.

\begin{corollary}\label{cor:rationalbases}
All but $o(\pi(X;4,(-1)^{k_1})^{d_K})$ ordered $d_K$-tuples of primes
$p_j\leq X$, $p_j\equiv(-1)^{k_1}\pmod4$, give a basis
\begin{align*}
 F_{p_1},\ldots,F_{p_{d_K}}
\end{align*}
of $S_K$ consisting of forms with rational Fourier coefficients.
\end{corollary}
\begin{proof}
Let $\HH_K=\{f_1,\ldots,f_{d_K}\}$ be the
normalized Hecke eigenbasis of $S_K$.
For $f\in\HH_K$, the Rankin--Selberg identity recalled in
\cite[Lemma 2.1]{Ni2026} gives
\begin{align}
 \langle f,F_p\rangle
 =C_pL(f,k_1+k_2+c-1)L(f,\chi_p,k_2+c),\label{eq:rankinselbergid}
\end{align}
where $C_p\neq0$ is independent of $f$. Since $k_1$ and $k_2$ have
the same parity, $k_2\geq k_1+2$, and hence
\begin{align*}
 k_2+c>(K+1)/2,\qquad k_1+k_2+c-1>(K+1)/2.
\end{align*}
Hence both $L$-values in \eqref{eq:rankinselbergid} are nonzero. Define
\begin{align*}
 A=(\langle f_i,F_{p_j}\rangle)_{1\leq i,j\leq d_K},\qquad
 M=(L(f_i,\chi_{p_j},k_2+c))_{1\leq i,j\leq d_K}.
\end{align*}
Since the $f_i$ span $S_K$, a form orthogonal to all of them must be zero.
Thus, for every column vector $\vect c=(c_1,\ldots,c_{d_K})^T$,
\begin{align*}
 A\overline{\vect c}=0
 \quad\Longleftrightarrow\quad
 \sum_{j=1}^{d_K}c_jF_{p_j}=0.
\end{align*}
Consequently, $A$ is nonsingular exactly when the $F_{p_j}$ are
linearly independent. Equation~\eqref{eq:rankinselbergid} gives
\begin{align*}
 A={}&\operatorname{diag}(L(f_i,k_1+k_2+c-1))_{i=1}^{d_K}
 \,M\,\operatorname{diag}(C_{p_j})_{j=1}^{d_K}.
\end{align*}
Note that both diagonal matrices are invertible, so $A$ is nonsingular exactly
when $M$ is. Corollary~\ref{cor:determinant}, applied in weight $K$,
proves the assertion.
\end{proof}


\section{Explicit bases in small dimensions}\label{sect:explicitbases}

Corollary~\ref{cor:determinant} gives a basis for almost every prime
tuple at each fixed point in the stated half-planes. When $\dim S_k=2$ or $3$,
the following theorem gives explicit conditions on the primes that
give a basis. 

\begin{theorem}\label{thm:explicitbases}
Suppose that $d=\dim S_k\in\{2,3\}$, and let $p_1,\ldots,p_d$
be odd primes satisfying the following conditions:
\begin{enumerate}
 \item If $d=2$, then $\chi_{p_1}(2)\neq\chi_{p_2}(2)$.
 \item If $d=3$, then $p_j>3$ for every $j$, and the three pairs
 \begin{align*}
  (\chi_{p_j}(2),\chi_{p_j}(3)),\qquad 1\leq j\leq3,
 \end{align*}
 are pairwise distinct.
\end{enumerate}
For $m\in\{1,2\}$ and every complex number $s$ with $\re(s)\geq\sigma_m+17$,
\begin{align}\label{eq:explicitdet}
 \det(L^{(m)}(f_i,\chi_{p_j},s))_{1\leq i,j\leq d}\neq0.
\end{align}
Thus $\mathcal L^{(m)}_{p_1,s},\ldots,\mathcal L^{(m)}_{p_d,s}$ form a basis of
$S_k^*$, and $K^{(m)}_{k,p_1,s},\ldots,K^{(m)}_{k,p_d,s}$ form a basis of $S_k$.
\end{theorem}

For $d=3$, the triples $(5,13,17)$ and $(7,11,19)$ satisfy
condition~(ii), with all primes congruent to $1$ and $3$ modulo $4$,
respectively.

For the following lemmas, let $d=\dim S_k\geq2$ and fix $m\in\{1,2\}$.
Put $w=s-\sigma_m+1$ and set
$u_i(r)=\lambda_{f_i}(r)$ if $m=1$, and $u_i(r)=b_{f_i}(r)$ if $m=2$.
For $\re(w)>1$, equations~\eqref{eq:standardproduct} and
\eqref{eq:reducedsym} give
\begin{align}\label{eq:reducedvalues}
 F_{i,p}(w)
 &:=\prod_r(1-u_i(r)\chi_p(r)r^{-w}
                     +\chi_p(r)^2r^{-2w})^{-1}\nonumber\\
 &=\begin{cases}
 L(f_i,\chi_p,s),&m=1,\\
L(\Sym^2f_i,\chi_p,s)/L(w,\chi_p),&m=2.
 \end{cases}
\end{align}
It follows that
\begin{align}\label{eq:reducedcoefficients}
 F_{i,p}(w)=\sum_{n\geq1}c_i(n)\chi_p(n)n^{-w},\qquad
 c_i(1)=1,\quad c_i(r)=u_i(r),\quad |c_i(n)|\leq\tau(n),
\end{align}
where $\tau(n)=\sum_{a\mid n}1$, and the series converges
absolutely for $\re(w)>1$.

\begin{lemma}\label{lem:smallcoefficientdet}
For $d=\dim S_k\in\{2,3\}$ and $m\in\{1,2\}$, put
\begin{align}\label{eq:smallcoefficientdet}
D_{k,m}=\det(c_i(j))_{1\leq i,j\leq d}.
\end{align}
Then $|D_{k,m}|\geq1/4$.
\end{lemma}
\begin{proof}
Write $t_k(n)=\Tr(T_n\mid S_k)$ for the trace of the $n$th Hecke operator on $S_k$,  so $t_k(1)=d$, and set
$G_{k,m}=(\Tr(T_{a^m}T_{b^m}))_{1\leq a,b\leq d}$.
The Hecke relation gives
\begin{align}\label{eq:traceproducts}
 (G_{k,m})_{ab}
 =\sum_{\ell\mid\gcd(a^m,b^m)}
   \ell^{k-1}t_k((ab)^m/\ell^2).
\end{align}
Since $T_nf_i=a_{f_i}(n)f_i$, we have $G_{k,m}=V^TV$, where
$V=(a_{f_i}(j^m))_{1\leq i,j\leq d}$.
Since each integer $2\leq j\leq d$ is prime, the identity
$a_{f_i}(j^2)=a_{f_i}(j)^2-j^{k-1}$ gives
\begin{align}\label{eq:tracecoefficients}
 c_i(j)=
 \begin{cases}
 a_{f_i}(j)/j^{(k-1)/2},&m=1,\\
 a_{f_i}(j^2)/j^{k-1}-1,&m=2.
 \end{cases}
\end{align}
Thus, $(c_i(j))_{1\leq i,j\leq d}$ is obtained from $V$ by dividing
column $j$ by $j^{m(k-1)/2}$ for $2\leq j\leq d$ and, when
$m=2$, subtracting the first column from each of these columns.
Consequently,
\begin{align}\label{eq:tracecertificate}
 D_{k,m}^2=\frac{\det G_{k,m}}{(d!)^{m(k-1)}}.
\end{align}

The dimension formula of $S_k$ \cite[Theorem 3.5.2]{diamondshurman} gives the weights in the table below.
We evaluate the traces in \eqref{eq:traceproducts} by the
Eichler--Selberg trace formula
\cite[Theorem~2, p.~48]{Zagier1976}.
Besides $t_k(1)=d$, we need only the traces $t_k(n)$ with $n=2,4,16$
when $d=2$, and with
$n=2,3,4,6,9,16,36,81$  when $d=3$.
Substituting these exact trace values into
\eqref{eq:traceproducts} and \eqref{eq:tracecertificate} gives
the following values of $\lfloor16D_{k,m}^2\rfloor$:
\begin{align}\label{eq:coefficienttable}
 \begin{array}{c|rr@{\qquad}c|rr}
 \multicolumn{3}{c}{d=2}&\multicolumn{3}{c}{d=3}\\
 k&m=1&m=2&k&m=1&m=2\\ \hline
 24&158&22&36&1088&1\\
 28&101&51&40&748&267\\
 30&56&7&42&634&12\\
 32&78&58&44&635&316\\
 34&91&156&46&452&47\\
 38&68&18&50&542&3
 \end{array}
\end{align}
Since every entry is at least $1$, we have $D_{k,m}^2\geq1/16$ and
$|D_{k,m}|\geq1/4$.
\end{proof}

\begin{lemma}\label{lem:determinantexpansion}
Suppose $d\geq2$. For any odd primes $p_1,\ldots,p_d$ and $w\in\CC$
with $\re(w)>1$,
\begin{align}\label{eq:determinantexpansion}
 \det(F_{i,p_j}(w))_{1\leq i,j\leq d}
 =\sum_{1\leq n_1<\cdots<n_d}
\frac{\det(c_i(n_\ell))_{1\leq i,\ell\leq d}
      \cdot \det(\chi_{p_j}(n_\ell))_{1\leq \ell,j\leq d}}
      {(n_1\cdots n_d)^w},
\end{align}
and the series converges absolutely.
\end{lemma}
\begin{proof}
Put $u=\re(w)>1$. For $N\geq d$, define the $d\times N$ and $N\times d$ matrices
\begin{align*}
 A_N&=(c_i(n)n^{-w})_{1\leq i\leq d,1\leq n\leq N},
 &C_N&=(\chi_{p_j}(n))_{\substack{1\leq n\leq N,1\leq j\leq d}},
\end{align*}
which gives
\begin{align*}
 M_N=A_NC_N
 =\left[\sum_{n\leq N}c_i(n)\chi_{p_j}(n)n^{-w}\right]_{1\leq i,j\leq d}.
\end{align*}
By the Cauchy--Binet formula \cite[p.~28]{Hornbook}, we have
\begin{align}\label{eq:finitecauchybinet}
 \det M_N
 =\sum_{1\leq n_1<\cdots<n_d\leq N}
 \frac{\det(c_i(n_\ell))_{1\leq i,\ell\leq d}
    \cdot\det(\chi_{p_j}(n_\ell))_{1\leq \ell,j\leq d}}
      {(n_1\cdots n_d)^w}.
\end{align}
Let $M=(F_{i,p_j}(w))_{1\leq i,j\leq d}$.
By \eqref{eq:reducedcoefficients} and $|n^{-w}|=n^{-u}$,
\begin{align*}
 \max_{i,j}|(M_N)_{ij}-M_{ij}|
 \leq\sum_{n>N}\frac{\tau(n)}{n^u}\longrightarrow0
 \qquad(N\to\infty),
\end{align*}
since $\sum_{n\geq1}\tau(n)n^{-u}=\zeta(u)^2<\infty$.
Continuity of the determinant therefore gives
\begin{align*}
 \lim_{N\to\infty}\det M_N=\det M.
\end{align*}
Hadamard's inequality \cite[p.~506]{Hornbook} and
$|\chi_{p_j}(n)|\leq1$ give
\begin{align}\label{eq:minorbound}
\left|\det(c_i(n_\ell))_{1\leq i,\ell\leq d}\cdot
       \det(\chi_{p_j}(n_\ell))_{1\leq \ell,j\leq d}\right|
 \leq d^d\prod_{\ell=1}^d\tau(n_\ell).
\end{align}
Consequently, for every integer $N\geq d$,
\begin{align}
 &\sum_{\substack{1\leq n_1<\cdots<n_d\nonumber\\ n_d>N}} \frac{\left|\det(c_i(n_\ell))_{1\leq i,\ell\leq d}
    \cdot\det(\chi_{p_j}(n_\ell))_{1\leq \ell,j\leq d}\right|}
      {(n_1\cdots n_d)^u}\\&\leq
 d^d\sum_{t>N}\frac{\tau(t)}{t^u}
       \sum_{n_1<\cdots<n_{d-1}<t}
       \prod_{\ell=1}^{d-1}\frac{\tau(n_\ell)}{n_\ell^u}\nonumber\\&
 \leq\frac{d^d\zeta(u)^{2d-2}}{(d-1)!}
\sum_{t>N}\frac{\tau(t)}{t^u},\label{eq:middlestep}
\end{align}
where we used the fact
\begin{align*}
 (d-1)!\sum_{n_1<\cdots<n_{d-1}<t}
       \prod_{\ell=1}^{d-1}\frac{\tau(n_\ell)}{n_\ell^u}
 &\leq \sum_{n_1,\ldots,n_{d-1}\geq1}
       \prod_{\ell=1}^{d-1}\frac{\tau(n_\ell)}{n_\ell^u}\\
 &=\left(\sum_{n\geq1}\frac{\tau(n)}{n^u}\right)^{d-1}
 \\&=\zeta(u)^{2d-2}.
\end{align*}
Taking $N=d$ in \eqref{eq:middlestep} proves absolute convergence
of the tail with $n_d>d$. The only remaining tuple is
$(1,\ldots,d)$, so the whole series in
\eqref{eq:determinantexpansion} converges absolutely.
As $N\to\infty$, the bound in \eqref{eq:middlestep} tends to zero,
so the finite sums in
\eqref{eq:finitecauchybinet} converge to the series in
\eqref{eq:determinantexpansion}. Together with
$\det M_N\to\det M$, this proves the identity.
\end{proof}

\begin{proof}[Proof of Theorem~\ref{thm:explicitbases}]
Put $w=s-\sigma_m+1$, $u=\re(w)$, and
$B=(\chi_{p_j}(n))_{1\leq n,j\leq d}$.
For $d=2$, condition~(i) gives $|\det B|=2$.
For $d=3$, its columns are $(1,\epsilon_j,\eta_j)^T$, where
the $(\epsilon_j,\eta_j)$ are three distinct elements of
$\{-1,1\}^2$. Every such choice gives $|\det B|=4$.

Separating the term $(n_1,\ldots,n_d)=(1,\ldots,d)$ from
the remaining terms in \eqref{eq:determinantexpansion}, we write
\begin{align}\label{eq:leadingdeterminant}
 \det(F_{i,p_j}(w))_{1\leq i,j\leq d}=D_{k,m}(\det B)(d!)^{-w}+E_d(w).
\end{align}
Let $Q_2=3$ and $Q_3=10$. Every nonzero term in $E_d(w)$ has
$n_1\cdots n_d\geq Q_d$. For $d=2$ this is immediate.
For $d=3$, the only remaining triple with product less than $10$
is $(1,2,4)$; the corresponding determinant $\det(\chi_{p_j}(n_\ell))_{1\leq\ell,j\leq3}$ vanishes because
$\chi_{p_j}(1)=\chi_{p_j}(4)=1$ for every $j$.
By \eqref{eq:minorbound}, for
$u\geq2$ we obtain
\begin{align*}
 |E_d(w)|
 &\leq d^dQ_d^{2-u}
 \sum_{n_1<\cdots<n_d}
       \prod_{\ell=1}^d\frac{\tau(n_\ell)}{n_\ell^2}\\
 &\leq\frac{d^d}{d!}Q_d^{2-u}\zeta(2)^{2d}.
\end{align*}
In particular,
\begin{align}\label{eq:determinanttails}
 |E_2(w)|\leq18\zeta(2)^4\,3^{-u},\qquad
 |E_3(w)|\leq450\zeta(2)^6\,10^{-u}.
\end{align}
For either $m\in\{1,2\}$, Lemma~\ref{lem:smallcoefficientdet}
gives $|D_{k,m}|\geq1/4$. Hence the leading term in
\eqref{eq:leadingdeterminant} satisfies
\begin{align*}
 |D_{k,m}(\det B)(d!)^{-w}|
 =|D_{k,m}\det B|(d!)^{-u}
 \geq\begin{cases}
 \frac12\,2^{-u},&d=2,\\
 6^{-u},&d=3.
 \end{cases}
\end{align*}
For $u\geq18$, the bounds in \eqref{eq:determinanttails} give
\begin{align*}
 \frac{|E_2(w)|}{\frac12\,2^{-u}}
 &\leq36\zeta(2)^4\left(\frac23\right)^u<1,\\
 \frac{|E_3(w)|}{6^{-u}}
 &\leq450\zeta(2)^6\left(\frac35\right)^u<1.
\end{align*}
Thus, the leading term in \eqref{eq:leadingdeterminant} strictly exceeds the error term $E_d(w)$ in absolute
value, and $\det(F_{i,p_j}(w))\neq0$ for
$u\geq18$, equivalently $\re(s)\geq\sigma_m+17$.
For $m=2$, the Euler product gives $L(w,\chi_{p_j})\neq0$
when $\re(w)>1$. Multiplying the $j$th column of
$(F_{i,p_j}(w))_{1\leq i,j\leq d}$ by this factor recovers
the matrix in \eqref{eq:explicitdet}.
The matrix $(\mathcal L^{(m)}_{p_j,s}(f_i))_{1\leq i,j\leq d}$
is therefore nonsingular, so these $d$ linear functionals form a
basis of $S_k^*$. The kernel representation \eqref{eq:kernelspectral}
then gives the corresponding basis of $S_k$.
\end{proof}
\section{Questions for future work}\label{sect:discussion}

In this section, we conclude the paper by proposing several questions for future work. The first natural question is to find effective versions of Theorems \ref{thm:standard} and \ref{thm:symmetric}.  
\begin{question}
For fixed $b\in\{1,3\}$, can one give a computable function
$\eta_{k,s}(X)\to0$ that bounds the left sides of
\eqref{eq:standarduniform} and \eqref{eq:symuniform} by
$\eta_{k,s}(X)P_b(X)$ in their respective ranges, together with
an explicit $X_0(\epsilon)$ such that
$\eta_{k,s}(X)<\epsilon$ for every $X\geq X_0(\epsilon)$?
\end{question}

The condition on the lower bound of $\re(s)$ in Theorem~\ref{thm:explicitbases} comes from the estimate for the error term $E_d(w)$ in 
\eqref{eq:determinanttails}. One may get a better lower bound by retaining more terms in
\eqref{eq:determinantexpansion}. A more difficult question is whether, at each fixed integer $\ell>\sigma_m$, every $d$-tuple of distinct odd primes gives a basis; see also \cite[Conjecture 1.4]{Ni2026}. This remains open even in dimensions two and three. 
\begin{question}
Suppose that $d=\dim S_k\in\{2,3\}$. Fix $m\in\{1,2\}$ and an
integer $\ell>\sigma_m$. Do we have
\begin{align*}
 \det(L^{(m)}(f_i,\chi_{p_j},\ell))_{1\leq i,j\leq d}\neq0
\end{align*}
for every $d$-tuple of distinct odd primes $p_1,\ldots,p_d$?
\end{question}

Nevertheless, the explicit criteria of Theorem~\ref{thm:explicitbases} also suggest the following extension.

\begin{question}
Fix $m\in\{1,2\}$ and let $d=\dim S_k\geq4$. Can one give
explicit conditions on finitely many symbols $\chi_{p_j}(r)$
such that 
\begin{align*}
 \det(L^{(m)}(f_i,\chi_{p_j},s))_{1\leq i,j\leq d}\neq0
 \qquad\text{for every real }s\geq s_0(k,m),
\end{align*}
with explicit bounds in terms of $k$ for $s_0(k,m)>\sigma_m$
and for the number and size of the primes $r$?
\end{question}

Another natural direction is to study the cases  $\re(s)=(k+1)/2$ for standard twists and $\re(s)=k$ for symmetric
square twists, where the uniform convergence in
Lemma~\ref{lem:uniformproduct} is no longer available. See \cite{Lamzouri2010} for the related results on the distribution of random Euler products. 
\section*{AI Declaration}
During the preparation of this article, ChatGPT was used for
literature searches, language editing, assistance with checking of proofs, and computational verification related to
Lemma \ref{lem:smallcoefficientdet}. The author prepared the final manuscript and takes full
responsibility for its content and correctness.
\bibliographystyle{amsplain}
\bibliography{ref}

@article {uniformweakconvergence1967,
    AUTHOR = {Billingsley, Patrick and Tops\o e, Flemming},
     TITLE = {Uniformity in weak convergence},
   JOURNAL = {Z. Wahrscheinlichkeitstheorie und Verw. Gebiete},
  FJOURNAL = {Zeitschrift f\"ur Wahrscheinlichkeitstheorie und Verwandte
              Gebiete},
    VOLUME = {7},
      YEAR = {1967},
     PAGES = {1--16},
   MRCLASS = {28.20 (60.00)},
  MRNUMBER = {209428},
MRREVIEWER = {H.\ Bergstr\"om},
       DOI = {10.1007/BF00532093},
       URL = {https://doi-org.libproxy.clemson.edu/10.1007/BF00532093},
}

@article {LoefflerIwasawa2019,
    AUTHOR = {Loeffler, David and Zerbes, Sarah Livia},
     TITLE = {Iwasawa theory for the symmetric square of a modular form},
   JOURNAL = {J. Reine Angew. Math.},
  FJOURNAL = {Journal f\"ur die Reine und Angewandte Mathematik. [Crelle's
              Journal]},
    VOLUME = {752},
      YEAR = {2019},
     PAGES = {179--210},
      ISSN = {0075-4102,1435-5345},
   MRCLASS = {11F67 (11F80 11G18 11R23)},
  MRNUMBER = {3975641},
MRREVIEWER = {Tobias\ Berger},
       DOI = {10.1515/crelle-2016-0064},
       URL = {https://doi-org.libproxy.clemson.edu/10.1515/crelle-2016-0064},
}

@article {Ni2026archive,
    AUTHOR = {Ni, Tianyu},
     TITLE = {A note on the critical values of symmetric square
              {$L$}-functions},
   JOURNAL = {Arch. Math. (Basel)},
  FJOURNAL = {Archiv der Mathematik},
    VOLUME = {126},
      YEAR = {2026},
    NUMBER = {1},
     PAGES = {53--60},
      ISSN = {0003-889X,1420-8938},
   MRCLASS = {11F67 (11F11)},
  MRNUMBER = {5004790},
       DOI = {10.1007/s00013-025-02186-y},
       URL = {https://doi-org.libproxy.clemson.edu/10.1007/s00013-025-02186-y},
}

@article {NiXueperiodandsymperiod2026,
    AUTHOR = {Ni, Tianyu and Xue, Hui},
     TITLE = {Periods and symmetric square periods of modular forms},
   JOURNAL = {Manuscripta Math.},
  FJOURNAL = {Manuscripta Mathematica},
    VOLUME = {177},
      YEAR = {2026},
    NUMBER = {4},
     PAGES = {Paper No. 51, 25},
      ISSN = {0025-2611,1432-1785},
   MRCLASS = {11F67 (11F11)},
  MRNUMBER = {5094731},
       DOI = {10.1007/s00229-026-01735-1},
       URL = {https://doi-org.libproxy.clemson.edu/10.1007/s00229-026-01735-1},
}

@book {diamondshurman,
    AUTHOR = {Diamond, Fred and Shurman, Jerry},
     TITLE = {A first course in modular forms},
    SERIES = {Graduate Texts in Mathematics},
    VOLUME = {228},
 PUBLISHER = {Springer-Verlag, New York},
      YEAR = {2005},
     PAGES = {xvi+436},
      ISBN = {0-387-23229-X},
   MRCLASS = {11Fxx},
  MRNUMBER = {2112196},
MRREVIEWER = {Henri\ Darmon},
}

@book {Hornbook,
    AUTHOR = {Horn, Roger A. and Johnson, Charles R.},
     TITLE = {Matrix analysis},
   EDITION = {Second},
 PUBLISHER = {Cambridge University Press, Cambridge},
      YEAR = {2013},
     PAGES = {xviii+643},
      ISBN = {978-0-521-54823-6},
   MRCLASS = {15-01},
  MRNUMBER = {2978290},
MRREVIEWER = {Mohammad\ Sal\ Moslehian},
}

@article {Stopplesymmetricsquare,
    AUTHOR = {Stopple, Jeffrey},
     TITLE = {Special values of twisted symmetric square {$L$}-functions and
              the trace formula},
   JOURNAL = {Compositio Math.},
  FJOURNAL = {Compositio Mathematica},
    VOLUME = {104},
      YEAR = {1996},
    NUMBER = {1},
     PAGES = {65--76},
      ISSN = {0010-437X,1570-5846},
   MRCLASS = {11F67 (11F25 11F41)},
  MRNUMBER = {1420711},
MRREVIEWER = {Glenn\ Stevens},
       URL = {http://www.numdam.org/item?id=CM_1996__104_1_65_0},
}

@article {rampantjantauni2026,
    AUTHOR = {Ni, Tianyu},
     TITLE = {Representation of {R}amanujan's tau function by twisted
              divisor functions},
   JOURNAL = {Int. J. Number Theory},
  FJOURNAL = {International Journal of Number Theory},
    VOLUME = {22},
      YEAR = {2026},
    NUMBER = {7},
     PAGES = {1525--1539},
      ISSN = {1793-0421,1793-7310},
   MRCLASS = {11F11 (11A25)},
  MRNUMBER = {5105657},
       DOI = {10.1142/S1793042126500788},
       URL = {https://doi-org.libproxy.clemson.edu/10.1142/S1793042126500788},
}

@article{LuoRamakrishnan1997,
  author = {Luo, Wenzhi and Ramakrishnan, Dinakar},
  title = {Determination of modular forms by twists of critical {$L$}-values},
  journal = {Invent. Math.},
  volume = {130},
  number = {2},
  year = {1997},
  pages = {371--398},
  doi = {10.1007/s002220050189},
  mrnumber = {1474162}
}

@article {Shimura1975,
    AUTHOR = {Shimura, Goro},
     TITLE = {On the holomorphy of certain {D}irichlet series},
   JOURNAL = {Proc. London Math. Soc. (3)},
  FJOURNAL = {Proceedings of the London Mathematical Society. Third Series},
    VOLUME = {31},
      YEAR = {1975},
    NUMBER = {1},
     PAGES = {79--98},
      ISSN = {0024-6115,1460-244X},
   MRCLASS = {10D15},
  MRNUMBER = {382176},
MRREVIEWER = {I.\ Piatetski-Shapiro},
       DOI = {10.1112/plms/s3-31.1.79},
       URL = {https://doi-org.libproxy.clemson.edu/10.1112/plms/s3-31.1.79},
}

@book{IwaniecKowalski2004,
  author = {Iwaniec, Henryk and Kowalski, Emmanuel},
  title = {Analytic number theory},
  series = {American Mathematical Society Colloquium Publications},
  volume = {53},
  publisher = {American Mathematical Society},
  address = {Providence, RI},
  year = {2004},
  doi = {10.1090/coll/053},
  mrnumber = {2061214}
}

@article{DukeKowalski2000,
  author = {Duke, William and Kowalski, Emmanuel},
  title = {A problem of {L}innik for elliptic curves and mean-value estimates for automorphic representations},
  journal = {Invent. Math.},
  volume = {139},
  number = {1},
  year = {2000},
  pages = {1--39},
  note = {With an appendix by Dinakar Ramakrishnan, ``Recovering modular forms from squares''},
  doi = {10.1007/s002229900017},
  url = {https://people.math.ethz.ch/~kowalski/linnik.pdf}
}

@article{ChintaDiaconu2005,
  author = {Chinta, Gautam and Diaconu, Adrian},
  title = {Determination of a {$\mathrm{GL}_3$} cuspform by twists of central {$L$}-values},
  journal = {Int. Math. Res. Not.},
  year = {2005},
  number = {48},
  pages = {2941--2967},
  doi = {10.1155/IMRN.2005.2941}
}

@article{KuanLesesvre2022,
  author = {Kuan, Chan Ieong and Lesesvre, Didier},
  title = {Quadratic twists of central values for {$\mathrm{GL}(3)$}},
  journal = {Q. J. Math.},
  volume = {73},
  number = {3},
  year = {2022},
  pages = {991--1034},
  doi = {10.1093/qmath/haab065}
}

@article{Lamzouri2010,
  author = {Lamzouri, Youness},
  title = {Distribution of values of {$L$}-functions at the edge of the critical strip},
  journal = {Proc. Lond. Math. Soc. (3)},
  volume = {100},
  number = {3},
  year = {2010},
  pages = {835--863},
  doi = {10.1112/plms/pdp050},
  url = {https://arxiv.org/abs/0810.2292},
  mrnumber = {2640292}
}

@article{KohamaMizuno2018,
  author = {Kohama, Hayato and Mizuno, Yoshinori},
  title = {Kernel functions of the twisted symmetric square of elliptic modular forms},
  journal = {Mathematika},
  volume = {64},
  number = {1},
  year = {2018},
  pages = {184--210},
  doi = {10.1112/S0025579317000389},
  mrnumber = {3778220}
}

@article{KohamaMizuno2019,
  author = {Kohama, Hayato and Mizuno, Yoshinori},
  title = {Corrigendum: Kernel functions of the twisted symmetric square of elliptic modular forms},
  journal = {Mathematika},
  volume = {65},
  number = {1},
  year = {2019},
  pages = {130--131},
  doi = {10.1112/S0025579318000359}
}

@article{Deligne1974,
  author = {Deligne, Pierre},
  title = {La conjecture de {W}eil. {I}},
  journal = {Inst. Hautes \'Etudes Sci. Publ. Math.},
  number = {43},
  year = {1974},
  pages = {273--307},
  doi = {10.1007/BF02684373},
  mrnumber = {0340258}
}

@book{Durrett2019,
  author = {Durrett, Rick},
  title = {Probability---theory and examples},
  edition = {Fifth},
  series = {Cambridge Series in Statistical and Probabilistic Mathematics},
  volume = {49},
  publisher = {Cambridge University Press},
  address = {Cambridge},
  year = {2019},
  doi = {10.1017/9781108591034},
  url = {https://sites.math.duke.edu/~rtd/PTE/PTE5_011119.pdf},
  mrnumber = {3930614}
}

@misc{Ni2026,
  author = {Ni, Tianyu},
  title = {Rational spanning sets of level-one cusp forms from {E}isenstein series at prime levels},
  year = {2026},
  note = {arXiv:2608.04241},
  url = {https://arxiv.org/abs/2608.04241}
}

@article{NiXueSymmetricPeriods2025,
  author = {Ni, Tianyu and Xue, Hui},
  title = {Linear independence of periods for the symmetric square {$L$}-functions},
  journal = {Ann. Math. Qu\'e.},
  volume = {49},
  number = {2},
  year = {2025},
  pages = {445--462},
  doi = {10.1007/s40316-025-00258-7}
}

@article{NiXueTwisted2026,
  author = {Ni, Tianyu and Xue, Hui},
  title = {Twisted periods of modular forms},
  journal = {Nagoya Math. J.},
  volume = {261},
  year = {2026},
  pages = {Paper No. e18, 39},
  doi = {10.1017/nmj.2026.10099},
  mrnumber = {5034067}
}

@book{CohenStromberg2017,
  author = {Cohen, Henri and Str{\"o}mberg, Fredrik},
  title = {Modular Forms: A Classical Approach},
  series = {Graduate Studies in Mathematics},
  volume = {179},
  publisher = {American Mathematical Society},
  address = {Providence, RI},
  year = {2017},
  doi = {10.1090/gsm/179}
}

@incollection{Zagier1976,
  author = {Zagier, Don},
  title = {The {Eichler--Selberg} trace formula on
           {$\mathrm{SL}_2(\mathbb{Z})$}},
  booktitle = {Introduction to Modular Forms},
  publisher = {Springer-Verlag},
  address = {Berlin--New York},
  series = {Grundlehren der mathematischen Wissenschaften},
  volume = {222},
  year = {1976},
  pages = {44--54},
  note = {Appendix to the book by Serge Lang}
}

\end{document}